\documentclass{article}
\usepackage[a4paper, margin=1.5in]{geometry}
\usepackage[dvipsnames]{xcolor}
\usepackage[T1]{fontenc}    
\usepackage[utf8]{inputenc}
\usepackage{microtype}
\usepackage{amsmath, amsthm, amssymb}
\usepackage{mathtools}
\usepackage[authoryear,round]{natbib}
\usepackage{comment,color}
\usepackage{tcolorbox}
\usepackage{hyperref}
    
\usepackage[capitalise]{cleveref}
\usepackage{bbm}

\usepackage[normalem]{ulem}

\newtheorem{theorem}{Theorem}
\newtheorem{proposition}[theorem]{Proposition}

\newtheorem{remark}[theorem]{Remark}

\theoremstyle{definition}
\newtheorem{definition}[theorem]{Definition}

\newcommand{\GP}{\mathcal{GP}}
\newcommand{\E}{\mathbb{E}}

\newcommand{\given}{\,|\,}
\newcommand{\bigo}{\mathcal{O}}

\newcommand{\R}{\mathbb{R}}

\newcommand{\bx}{\mathbf{x} }

\hypersetup{
    colorlinks=true,
    linkcolor=blue,
    filecolor=magenta,      
    urlcolor=cyan,
    citecolor=blue
    }

\def\argmin{\mathop{\text{arg\,min}}}

\newcommand{\prt}{\mathcal{P}}

\newcommand{\loss}{\mathcal{L}}

\def\sigmasqicvest{\hat{\sigma}_\textup{ICV}^2}
\def\sigmasqcvest{\hat{\sigma}_\textup{CV}^2}
\def\sigmasqmlest{\hat{\sigma}_\textup{ML}^2}

\title{Comparing Scale Parameter Estimators for Gaussian Process Interpolation with the Brownian Motion Prior: Leave-One-Out Cross Validation and Maximum Likelihood\footnote{Published in SIAM/ASA Journal on Uncertainty Quantification (2025), 13(2), pp. 679--717. \newline\hspace*{15.5pt}\href{https://doi.org/10.1137/23M1586884}{https://doi.org/10.1137/23M1586884}}}

\author{
  Masha Naslidnyk\thanks{Department of Statistical Science, University College London, UK.}
  \and
  Motonobu Kanagawa\thanks{Data Science Department, EURECOM, Biot, France.}
  \and
  Toni Karvonen\thanks{Department of Mathematics and Statistics, University of Helsinki, Finland.}
  \and
  Maren Mahsereci\thanks{Yahoo Research, Berlin, Germany. Work performed while at the University of Tübingen, Germany.}
}
\date{}

\begin{document}
% !TEX pdfSinglePage
\maketitle

\begin{abstract}

Gaussian process (GP) regression is a Bayesian nonparametric method for regression and interpolation, offering a principled way of quantifying the uncertainties of predicted function values. For the quantified uncertainties to be well-calibrated, however, the kernel of the GP prior has to be carefully selected. In this paper, we theoretically compare two methods for choosing the kernel in GP regression: cross-validation and maximum likelihood estimation. Focusing on the scale-parameter estimation of a Brownian motion kernel in the noiseless setting, we prove that cross-validation can yield asymptotically well-calibrated credible intervals for a broader class of ground-truth functions than maximum likelihood estimation, suggesting an advantage of the former over the latter.  Finally, motivated by the findings, we propose \emph{interior cross validation}, a procedure that adapts to an even broader class of ground-truth functions.

\end{abstract}

\tableofcontents

\section{Introduction} \label{sec:introduction}

Gaussian process (GP) regression (or kriging) is a Bayesian nonparametric method for regression and interpolation that has been extensively studied in statistics and machine learning \citep{OHagan1978, Stein1999, rassmussen2006gaussian}. Its key property is that it enables uncertainty quantification of estimated function values in a principled manner, which is crucial for applications involving decision-making, safety concerns, and scientific discovery. As such, GP regression has been a core building block of more applied algorithms, including Bayesian optimisation \citep{jones1998efficient,shahriari2015taking,garnett2023bayesian}, probabilistic numerical computation \citep{hennig2015probabilistic, cockayne2019bayesian, pnbook2022}, and calibration and emulation of computer models~\citep{sacks1989design,kennedy2001bayesian,o2006bayesian, beck2016sequential, gu2018robust}, to name just a few.

GP regression estimates an unknown function $f$ from its observations as follows. First, a {\em prior distribution} for $f$ is defined as a GP, by specifying its {\em kernel}, and mean function.
Given $N$ observations of $f$, the {\em posterior distribution} of $f$ is derived, which is another GP with mean function $m_N$ and kernel (or covariance function) $k_N$.
The function value $f(x)$ at any input $x$ can then be predicted as the posterior mean $m_N(x)$, and its uncertainty is quantified using the posterior standard deviation $\sqrt{\smash[b]{k_N(x)}} \coloneqq \sqrt{ \smash[b]{k_N(x,x) }}$.
Specifically, a {\em credible interval} of $f(x)$ can be constructed as the interval $[ m_N(x) - \alpha \sqrt{\smash[b]{k_N(x)}},  m_N(x) + \alpha \sqrt{ \smash[b]{k_N(x)}} ]$ for a constant $\alpha > 0$ (for example, $\alpha \approx 1.96$ leads to the 95\% credible interval). Such uncertainty estimates
constitute key ingredients in the above applications of GP regression.

For GP uncertainty estimates to be reliable, the posterior standard deviation $\sqrt{\smash[b]{k_N(x)}}$ should, ideally, decay at the {\em same} rate as the prediction error $|m_N(x) - f(x)|$ decreases, with the increase of sample size $N$. Otherwise, GP uncertainty estimates are either asymptotically {\em overconfident} or {\em underconfident}.  For example, if  $\sqrt{\smash[b]{k_N(x)}}$ goes to $0$ faster than the error $|m_N(x) - f(x)|$, then the credible interval $[m_N(x) - \alpha \sqrt{k_N(x)}, m_N(x) + \alpha \sqrt{\smash[b]{k_N(x)}}  ] $ will {\em not} contain the true value $f(x)$ as $N$ increases for {\em any} fixed constant $\alpha > 0$ (asymptotically overconfident). If  $\sqrt{\smash[b]{k_N(x)}}$ goes to $0$ slower than the error $|m_N(x) - f(x)|$, then the confidence interval $[m_N(x) - \alpha \sqrt{\smash[b]{k_N(x)}}, m_N(x) + \alpha \sqrt{\smash[b]{k_N(x)}}  ] $ will get larger than the error $|m_N(x) - f(x)|$ as $N$ increases (asymptotically underconfident). Both of these cases are not desirable in practice, as GP credible intervals will not be accurate estimates of prediction errors.

\begin{figure}[t]
    \centering
    \includegraphics[width=\linewidth]{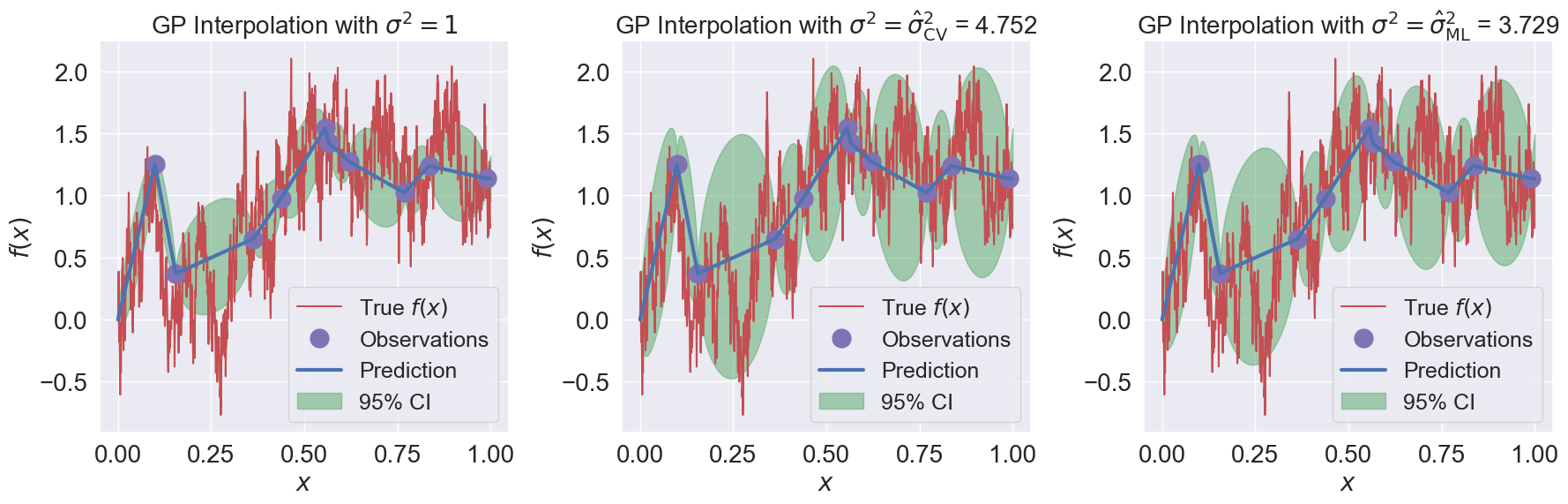}
    \caption{GP interpolation of a fractional Brownian motion with the Hurst parameter $H = 0.2$ (smoothness $l+\alpha = 0.2$) using the Brownian motion kernel~\eqref{eq:bm-kernel-intro} with three different scale parameters: $\sigma^2 = 1$ (left),  $\sigma^2 = \hat{\sigma}_{\rm CV}^2 = 4.752$ given by the LOO-CV estimator (middle) and  $\sigma^2 = \hat{\sigma}_{\rm ML}^2 = 3.729$ obtained with the ML estimator (right). In each figure, the red trajectory represents the path of the fractional Brownian motion, the purple circles the training data, the blue curve the posterior mean $m_N(x)$ and the green shade the 95 \% credible interval $[ m_N(x) - 1.96\sigma \sqrt{\smash[b]{k_N(x)}},  m_N(x) + 1.96\sigma \sqrt{\smash[b]{k_N(x)}} ]$.}
    \label{fig:FBM-H02}
\end{figure}

\begin{figure}[t]
    \centering
    \includegraphics[width=\linewidth]{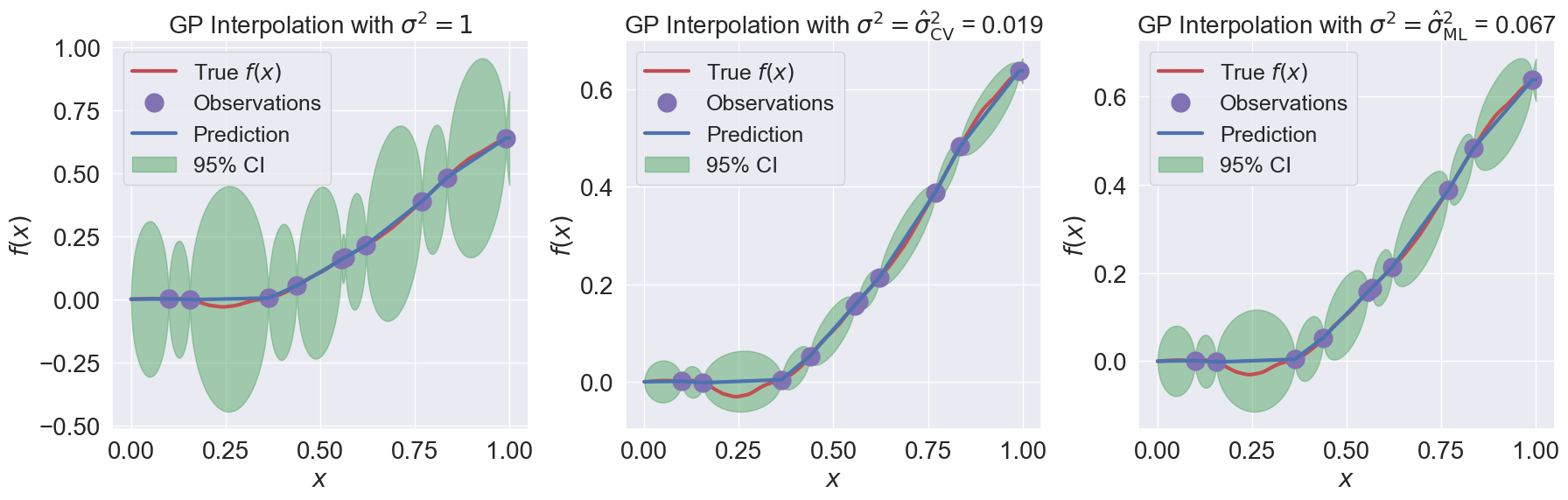}
    \caption{GP interpolation of an integrated fractional Brownian motion with the Hurst parameter $H = 0.5$ (smoothness $l+\alpha = 1.5$) using the Brownian motion kernel~\eqref{eq:bm-kernel-intro}
 with three different scale parameters: $\sigma^2 = 1$ (left),  $\sigma^2 = \hat{\sigma}_{\rm CV}^2 = 0.019$ given by the LOO-CV estimator (middle) and  $\sigma^2 = \hat{\sigma}_{\rm ML}^2 = 0.067$ obtained with the ML estimator (right). For the explanation of the figures, see the caption of Figure \ref{fig:FBM-H02}.}
    \label{fig:IFBM-H05}
\end{figure}

Unfortunately, in general, the posterior standard deviation $\sqrt{\smash[b]{k_N(x)}}$ does {\em not} decay at the same rate as the prediction error $| f(x) - m_N(x) |$, because, as is well-known, $\sqrt{\smash[b]{k_N(x)}}$ does {\em not} depend on the true function $f$; see \eqref{eq:gp-cov-func} in Section~\ref{sec:GP-regres}. Exceptionally, if the function $f$ is a sample path of the GP prior (the well-specified case),  GP uncertainty estimates can be well-calibrated. However, in general, the unknown $f$ is not exactly a sample path of the GP prior (the misspecified case), and the posterior standard deviation $\sqrt{\smash[b]{k_N(x)}}$ does not scale with the prediction error $| f(x) - m_N(x) |$. Figures \ref{fig:FBM-H02} and \ref{fig:IFBM-H05}  (the left panels) show examples where the true function $f$ is not a sample of the GP prior and where the GP uncertainty estimates are not well-calibrated.

\subsection{Scale Parameter Estimation}

To obtain sensible uncertainty estimates, it is therefore
necessary to adapt the posterior standard deviation $\sqrt{\smash[b]{k_N(x)}}$ to the function $f$. One simple way to achieve this is to introduce the {\em scale parameter} $\sigma^2 > 0$ and parametrise the kernel as
\begin{equation} \label{eq:scale-kernel}
k_\sigma(x,x') \coloneqq \sigma^2 k(x,x'),
\end{equation}
where $k$ is the original kernel. GP regression with this kernel $k_\sigma$ yields the posterior mean function $m_N$, which is not influenced by $\sigma^2$, and the posterior covariance function $\sigma^2 k$, which is scaled by $\sigma^2$. If $\sigma^2$ is estimated from observed data of $f$, the estimate $\hat{\sigma}^2$ depends on $f$, and so does the resulting posterior standard deviation $\hat{\sigma} \sqrt{\smash[b]{k_N(x)}}$.

 One approach to scale-parameter estimation is the method of {\em maximum likelihood (ML)}, which optimises $\sigma^2$ to maximise the marginal likelihood of the GP \citep[Section 5.4]{rassmussen2006gaussian}. The ML approach is popular for general hyperparameter optimisation in GP regression. Another less common way in the GP literature is {\em cross-validation (CV)}, which optimises $\sigma^2$ to maximise the average predictive likelihood with held-out data~\citep{sundararajan2001predictive}. For either approach, the optimised scale parameter can be obtained analytically in computational complexity $\bigo(N^3)$.
Figures \ref{fig:FBM-H02} and \ref{fig:IFBM-H05} (middle and right panels) demonstrate that both approaches yield uncertainty estimates better calibrated than the original estimates without the scale parameter.

Do these scale parameter estimators lead to asymptotically well-calibrated uncertainty estimates? To answer this question, it is necessary to understand their convergence properties as the sample size $N$ increases.
Most existing theoretical works focus on the well-specified case where there is a ``true'' scale parameter $\sigma_0^2$ such that the unknown $f$ is a GP with the kernel $\sigma_0^2 k$. In this case, both the ML and CV estimators have been shown to be consistent in estimating the true $\sigma_0^2$ \citep[e.g.,][]{Ying1991, Zhang2004,Bachoc2017, bachoc2020asymptotic}.
However, in general, no ``true'' scale parameter $\sigma^2_0$ exists such that the unknown $f$ is a GP with the covariance $\sigma_0^2 k$. In such misspecified cases, not much is known about the convergence properties of both estimators.
\citet{Karvonen2020} analyse the ML estimator for the scale parameter, assuming that $f$ is a deterministic function. They derive upper bounds (and lower bounds in some cases) for the ML estimator; see \citet{Wang2021} for closely related work. To our knowledge, no theoretical work exists for the CV estimator for the scale parameter in the misspecified case. \cite{Bachoc2013} and \cite{petit2021gaussian} empirically compare the ML and CV estimators under different model misspecification settings. We will review other related works in Section~\ref{sec:related-work}.

\subsection{Contributions}

This work studies the convergence properties of the ML and CV estimators, $\hat{\sigma}_{\rm ML}^2$ and $\hat{\sigma}_{\rm CV}^2$, of the scale parameter $\sigma^2$ in GP regression, to understand whether they lead to asymptotically well-calibrated uncertainty estimates. In particular, we provide the first theoretical analysis of the CV estimator $\hat{\sigma}_{\rm CV}^2$ when the GP prior is misspecified, and also establish novel results for the ML estimator $\hat{\sigma}_{\rm ML}^2$.

To facilitate the analysis, we focus on the following simplified setting. For a constant $T > 0$, let $[0, T] \subset \mathbb{R}$ be the input domain. Let $k$ in \eqref{eq:scale-kernel} be the Brownian motion kernel
\begin{equation}
\label{eq:bm-kernel-intro}
    k(x, x') = \min(x, x') \quad \text{ for } \quad x, x' \in [0, T].
\end{equation}
With this choice, a sample path of the GP prior, roughly speaking, has smoothness of 1/2 (in terms of the differentiability; we will be more rigorous in later sections).

We assume that the true unknown function $f$ has the smoothness $l + \alpha$, where $l \in \{0\} \cup \mathbb{N}$ and $0 < \alpha \leq 1$.
The GP prior has well-specified smoothness if $l = 0$ and $\alpha = 1/2$. Other settings of $l$ and $\alpha$ represent misspecified cases. If $l = 0$ and $\alpha <1/2$, the true function $f$ is rougher than the GP prior (Figure \ref{fig:FBM-H02}); if $l = 0$ and $\alpha > 1/2$ or $l \geq 1$, the function $f$ is smoother than the GP prior. We focus on the noise-free setting where the function values $f(x_1), \dots, f(x_N)$ are observed at input points $x_1, \dots, x_N \in [0,T]$.

Our main results are new upper and lower bounds for the asymptotic rates of the CV estimator $\hat{\sigma}_{\rm CV}^2$ and the ML estimator $\hat{\sigma}_{\rm ML}^2$ as $N \to \infty$ (Section \ref{sec:limit-behaviour-for-sigma}). The results suggest that the CV estimator can yield asymptotically well-calibrated uncertainty estimates for a broader class of functions $f$ than the ML estimator; thus, the former has an advantage over the latter (Section \ref{sec:discussion}).   More specifically, asymptotically well-calibrated uncertainty estimates may be obtained with the CV estimator for the range $0 < l + \alpha \leq 3/2$ of smoothness of the true function, while this range becomes $0 < l + \alpha \leq 1$ with the ML estimator and is narrower. This finding is consistent with the example in Figure~\ref{fig:IFBM-H05}, where the true function has smoothness $l + \alpha = 3/2$ and is thus smoother than the GP prior. The uncertainty estimates of the CV estimator appear to be well-calibrated, while those of the ML estimator are unnecessarily wide, failing to adapt to the smoothness. Motivated by these insights, we propose a method called \emph{interior cross-validation}, and show it accommodates an even wider range of smoothness of the true function than the CV estimator.

This paper is structured as follows. After reviewing related works in
\Cref{sec:related-work}, we introduce the necessary background on the ML and CV approaches to scale parameter estimation for GP regression in \Cref{sec:par-est-for-gps}. We describe the setting of the theoretical analysis in \Cref{sec:setting},   present our main results in  \Cref{sec:limit-behaviour-for-sigma}, and discuss its consequences on uncertainty quantification in  \Cref{sec:discussion}. We report simulation experiments in \Cref{sec:experiments}, conclude in \Cref{sec:conclusion}, and present proofs in \Cref{sec:proofs}.

\subsection{Related work}
\label{sec:related-work}

We review here related theoretical works on hyper-parameter selection in GP regression in the noiseless setting. We categorise them into two groups based on how the true unknown function $f$ is modelled: random and deterministic.

\paragraph{Random setting.}

One group of works models the ground truth $f$ as a random function, specifically as a GP. Most of these works model $f$ as a GP with a Mat\'ern-type kernel and analyse the ML estimator. Under the assumption that the GP prior is correctly specified,  asymptotic properties of the ML estimator for the scale parameter and other parameters have been studied \citep{Stein1990, Ying1991, Ying1993, LohKam2000, Zhang2004, Loh2005, Du2009, Anderes2010, WangLoh2011, Kaufman2013,Bevilacqua2019}.
Recently \citet{LohSunWen2021} and \citet{LohSun2023} have constructed consistent estimators of various parameters for many commonly used kernels, including Mat\'erns.  \citet{Chen2021} and \citet{Petit2023} consider a periodic version of Mat\'ern GPs,  and show the consistency of the ML estimator for its smoothness parameter. To our knowledge, the only existing theoretical result for ML estimation of the scale parameter in the misspecified random setting considers oversmoothing~\citep[Theorem 4.2]{karvonen2021estimation}. Oversmoothing refers to the situation where the chosen kernel is smoother than the true function. In~\Cref{sec:random-setting} (Theorem~\ref{res:holder-spaces-exp-ml}), we provide a result for the undersmoothing case, which occurs when the chosen kernel is less smooth than the true function.

In contrast, few theoretical works exist for the CV estimator.  \citet{Bachoc2017} study the leave-one-out (LOO) CV estimator for the Mat\'ern-1/2 model (or the Laplace kernel) with one-dimensional inputs, in which case the GP prior is an Ornstein--Uhlenbeck (OU) process. Assuming the well-specified case where the true function is also an OU process, they prove the consistency and asymptotic normality of the CV estimator for the microergodic parameter in the fixed-domain asymptotic setting. \citet{bachoc2018asymptotic} and \citet{bachoc2020asymptotic} discuss another CV estimator that uses the mean square prediction error as the scoring criterion of CV (thus different from the one discussed here) in the increasing-domain asymptotics.   \citet{Bachoc2013} and \citet{petit2021gaussian}  perform empirical comparisons of the ML and CV estimators under different model misspecification settings. Thus, to our knowledge, no theoretical result exists for the CV estimator of the scale parameter in the random misspecified setting, which we provide in \Cref{sec:random-setting} (Theorem~\ref{res:holder-spaces-exp}).

\paragraph{Deterministic setting.}
Another line of research assumes that the ground truth $f$ is a fixed function belonging to a specific function space \citep{Stein1993}. \citet{XuStein2017} assumed that the ground truth $f$ is a monomial on $[0,1]$ and proved some asymptotic results for the ML estimator when the kernel $k$ is Gaussian.
As mentioned earlier, \citet{Karvonen2020} proved asymptotic upper (and, in certain cases, also lower) bounds on the ML estimator $\hat{\sigma}_{\rm ML}^2$ of the scale parameter $\sigma^2$; see \citet{Wang2021} for a closely related work. \citet{Karvonen2022} has studied the ML and LOO-CV estimators for the smoothness parameter in the Mat\'ern model; see also  \citet{Petit2023}. \citet{BenSalem2019} and \citet{KarvonenOates2023} proved non-asymptotic results on the length-scale parameter in the Mat\'ern and related models. Thus, there has been no work for the CV estimator of the scale parameter $\sigma^2$ in the deterministic setting, which we provide in Section \ref{sec:results-deterministic} (Theorem \ref{res:holder-spaces}); we also prove a corresponding result for the ML estimator (Theorem \ref{res:holder-spaces-ml}).

\section{Background}
\label{sec:par-est-for-gps}

This section briefly reviews GP regression and the ML and LOO-CV estimators of kernel parameters.

\subsection{Gaussian process regression}
\label{sec:GP-regres}

We first explain GP regression (or interpolation). Let $\Omega$ be a set, and $f\colon \Omega \to \mathbb{R}$ be an unknown function of interest. Suppose $N$ function values $f(x_1), \dots, f(x_N)$ are observed at pairwise distinct input points $x_1, \dots, x_N \in \Omega$.
The task here is to estimate $f$ based on the data $(\bx, f(\bx))$, where  $f(\bx) \coloneqq [f(x_1), \dots, f(x_N)]^\top \in \mathbb{R}^N$  and $\bx \coloneqq [x_1, \ldots, x_N]^\top \in \Omega^N$.

In GP regression, first a prior distribution of the unknown $f$ is defined as a GP by specifying its mean function $m \colon \Omega \to \R$ and covariance function (kernel) $k \colon \Omega \times \Omega \to \R$; we may write $f \sim \GP(m,k)$ to indicate this. Conditioned on the data $(\bx, f(\bx))$, the posterior distribution of $f$ is again a GP whose mean function $m_N: \Omega \to \mathbb{R}$ and covariance function $k_N: \Omega \times \Omega \to \mathbb{R}$ are given by
\begin{subequations}
\label{eq:posterior-moments}
  \begin{align}
    m_N(x) & \coloneqq m(x) + k(x, \bx)^\top k(\bx, \bx)^{-1} \left(f(\bx) - m(\bx)\right), \quad x \in \Omega, \\
    k_N(x, x') & \coloneqq k(x, x') - k(x, \bx)^\top k(\bx, \bx)^{-1} k(x', \bx), \quad x, x' \in \Omega, \label{eq:gp-cov-func}
  \end{align}
\end{subequations}%
where $m(\bx) \coloneqq [m(x_1), \ldots, m(x_N)]^\top \in \R^N$ and $k(x, \bx) \coloneqq [k(x, x_1), \ldots, k(x, x_N)]^\top \in \R^N$, and
\begin{equation}
\label{eq:gram-matrix}
    k(\bx, \bx) \coloneqq
    \begin{bmatrix}
    k(x_1, x_1) & \dots & k(x_1, x_N) \\
    \vdots & \ddots & \vdots \\
    k(x_N, x_1) & \dots & k(x_N, x_N)
    \end{bmatrix} \in \mathbb{R}^{N \times N}
\end{equation}%
is the Gram matrix.
Throughout this paper, we assume that the points $\bx$ are such that the Gram matrix is non-singular.
For notational simplicity, we may write the posterior variance as
\begin{equation*}
  k_N(x) \coloneqq k_N(x, x), \quad x \in \Omega.
\end{equation*}
For simplicity and as commonly done, we henceforth assume that the prior mean function $m$ is the zero function, $m(\cdot) \equiv 0$.

While the GP prior assumes that the unknown function $f$ is a sample path of the GP with the specified kernel $k$, this assumption does not hold in general, i.e., model misspecification occurs. In this case, as described in Figures \ref{fig:FBM-H02} and \ref{fig:IFBM-H05} (left), the posterior standard deviation $\sqrt{\smash[b]{k_N(x)}}$, which is supposed to quantify the uncertainty of the unknown function value $f(x)$,  may not be well calibrated with the prediction error $| m_N(x) - f(x) |$. This issue can be addressed by selecting the kernel $k$ or its parameters from the data $( \bx, f(\bx) )$; we will explain this topic next.

\subsection{Kernel parameter estimation}
\label{sec:kernel-parameter-estimation}

The selection of the kernel $k$ is typically performed by defining a parametric family of kernels $\{k_\theta \}_{\theta \in \Theta}$ and selecting the parameter $\theta $ based on an appropriate criterion.  Here $\Theta$ is a parameter set, and $k_\theta: \Omega \times \Omega \to \mathbb{R}$ for each $\theta \in \Theta$ is a kernel.

\paragraph{Maximum likelihood (ML) estimation.}
The ML estimator maximises the log-likelihood of the GP $f$ with kernel  $k_\theta$ under the data $(\bx, f(\bx))$:
\begin{equation*}
  \log p(f(\bx) \given \bx, \theta) = -\frac{1}{2} \bigg( f(\bx)^\top k_\theta(\bx, \bx)^{-1} f(\bx) + \log \det k_\theta(\bx, \bx) + n \log (2\pi) \bigg),
\end{equation*}
where $\det k_\theta(\bx, \bx)$ is the determinant of the Gram matrix $k_\theta(\bx, \bx)$ (see, e.g., \citealt[Section~5.4.1]{rassmussen2006gaussian}).
With the additive terms that do not depend on $\theta$ removed from $\log p(f(\bx) \given \bx, \theta)$, this is equivalent to minimising the loss function
\begin{equation} \label{eq:ML-loss-function}
    \loss_\textup{ML}(\theta) := f(\bx)^\top k_\theta (\bx, \bx)^{-1} f(\bx) + \log \det k_\theta (\bx, \bx).
\end{equation}
In general, $\loss_\textup{ML}(\theta)$ may not have a unique minimiser, so that any ML estimator satisfies
\begin{equation*}
     \hat{\theta}_\textup{ML} \in \argmin_{\theta \in \Theta}\loss_\textup{ML}(\theta).
\end{equation*}

\paragraph{Leave-one-out cross-validation (LOO-CV).}
The LOO-CV estimator \citep[e.g.,][Section~5.4.2]{rassmussen2006gaussian}, which we may simply call the CV estimator, is an alternative to the ML estimator. It maximises the average log-predictive likelihood
\begin{equation} \label{eq:cv-estimator-any-kernel}
  \sum_{n=1}^N \log p( f(x_n) \given x_n, \bx_{\setminus n}, f(\bx_{\setminus n}), \theta)
\end{equation}
with held-out data $(x_n, f(x_n))$, where $n = 1, \dots, N$, based on the data $(\bx_{\setminus n}, f(\bx_{\setminus n}))$, where $\bx_{\setminus n}$ denotes the input points with $x_n$ removed:
\begin{equation*}
  \bx_{\setminus n} = [x_1, \dots, x_{n-1}, x_{n+1}, \dots , x_N]^\top \in \Omega^{N - 1}.
\end{equation*}
Let $m_{\theta, \setminus  n}$ and $ k_{\theta, \setminus  n}$ denote the posterior mean and covariance functions of GP regression with the kernel $k_\theta$ and the data $(\bx_{\setminus n}, f(\bx_{\setminus n})$.
Because each $p( f(x_n) \given x_n, \bx_{\setminus n}, f(\bx_{\setminus n}), \theta)$ is the Gaussian density of $f(x_n)$ with mean $m_{\theta, \setminus  n}(x_n)$ and variance $ k_{\theta, \setminus  n}(x_n) \coloneqq k_{\theta, \setminus  n}(x_n,x_n)$, removing additive terms that do not depend on $\theta$ and reversing the sign in \eqref{eq:cv-estimator-any-kernel} yields the following CV objective function:
\begin{equation} \label{eq:CV-loss-function}
    \loss_\textup{CV}(\theta) = \sum_{n=1}^N \frac{\left[f(x_n) - m_{\theta, \setminus n}(x_n)\right]^2}{ k_{\theta, \setminus  n}(x_n)}  + \log k_{\theta, \setminus  n} (x_n).
\end{equation}%
The CV estimator is then defined as its minimiser:
\begin{equation*}
  \hat{\theta}_\textup{CV} \in \argmin_{\theta \in \Theta}\loss_\textup{CV}(\theta).
\end{equation*}
As for the ML estimator, the CV objective function and its first-order gradients can be computed in closed form in $\bigo(N^3)$  time~\citep{sundararajan2001predictive}.

\paragraph{Scale parameter estimation.}
As explained in Section \ref{sec:introduction}, we consider the family of kernels $k_\sigma (x,x') \coloneqq \sigma^2 k(x,x')$ parametrised with the scale parameter $\sigma^2>0$, where $k$ is a fixed kernel, and study the estimation of $\sigma^2$ using the CV and ML estimators, denoted as $\hat{\sigma}_{\rm CV}^2$ and  $\hat{\sigma}_{\rm ML}^2$, respectively. In this case, both $\hat{\sigma}_{\rm ML}^2$ and  $\hat{\sigma}_{\rm CV}^2$ can be derived in closed form by differentiating~\eqref{eq:ML-loss-function} and~\eqref{eq:CV-loss-function}.

Let $m_{n-1}$ and $k_{n-1}$ be the posterior mean and variance functions of GP regression using the kernel $k$ and the first $n-1$ training observations $(x_1, f(x_1)), \dots, (x_{n-1}, f(x_{n-1}))$. Let $m_0(\cdot) \coloneqq 0$ and $k_0(x,x) \coloneqq k(x,x)$. Then the ML estimator is given by
\begin{equation}
\label{eq:sigma-ml}
  \sigmasqmlest = \frac{f(\bx)^\top k(\bx, \bx)^{-1} f(\bx)}{N} = \frac{1}{N} \sum_{n=1}^N \frac{[ f(x_n) - m_{n-1}(x_n) ]^2}{k_{n-1}(x_n)}.
\end{equation}
This expression of the ML estimator is relatively well-known; see e.g.~Section~4.2.2 in \citet{XuStein2017} or Proposition~7.5 in \citet{KarvonenOates2023}.

On the other hand, the CV estimator $\hat{\sigma}_{\rm CV}^2$ is given by
\begin{align}
\label{eq:sigma-cv}
    \sigmasqcvest = \frac{1}{N} \sum_{n=1}^N  \frac{\left[f(x_n) - m_{\setminus  n}(x_n)\right]^2}{ k_{\setminus  n}(x_n)},
\end{align}%
where  $m_{\backslash n}$ and $k_{\backslash n}$ are the posterior mean and covariance functions of GP regression using the kernel $k$ and data $(\bx_{\setminus n}, f(\bx_{\setminus n}))$ with $(x_n, f(x_n))$ removed:
\begin{align*}
    m_{\setminus n}(x) &= k(\bx_{\setminus n}, x)^\top k(\bx_{\setminus n}, \bx_{\setminus n})^{-1} f(\bx_{\setminus n}), \\
    k_{\setminus n}(x, x') &= k(x, x') -  k(\bx_{\setminus n}, x)^\top k(\bx_{\setminus n}, \bx_{\setminus n})^{-1} k(\bx_{\setminus n}, x').
\end{align*}%

Notice the similarity between the two expressions \eqref{eq:sigma-ml} and \eqref{eq:sigma-cv}. The difference is that the ML estimator uses $k_{n-1}$ and $m_{n-1}$, which are based on the first $n-1$ training observations, while the CV estimator uses  $k_{\backslash n}$ and $m_{\backslash n}$ obtained with $N-1$ observations, for each $n = 1,\dots, N$. Therefore, the CV estimator uses all the data points more evenly than the ML estimator. This difference may be the source of the difference in their asymptotic properties established later.

\begin{remark}
  \label{remark:ml-cv}
  As suggested by the similarity between \eqref{eq:sigma-ml} and~\eqref{eq:sigma-cv}, there is a deeper connection between ML and CV estimators in general. For instance, \citet[Proposition~2]{FongHolmes2020} have shown that the Bayesian marginal likelihood equals the average of leave-$p$-out CV scores. We prove this result for the special case of scale parameter estimation in GP regression in \Cref{sec:proofs-ml-cv-relation}. Another notable example is the work in~\citet{ginsbourger2021fast}, where the authors showed that, when corrected for the covariance of residuals, the CV estimator of the scale parameter reverts to MLE.
\end{remark}

\section{Setting}
\label{sec:setting}

This section describes the settings and tools for our theoretical analysis: the Brownian motion kernel in Section \ref{sec:bm-kernel-setting}; sequences of partitions in Section \ref{sec:sequence-of-partitions}; the H\"older class of functions in Section \ref{sec:holder-space}; fractional Brownian motion in Section \ref{sec:fbm}; and functions of finite quadratic variation in Section \ref{sec:quadratic-variation}.

\subsection{Brownian motion kernel}
\label{sec:bm-kernel-setting}

As explained in Section \ref{sec:introduction}, for the kernel $k$ we focus on the Brownian motion kernel on the domain $\Omega = [0, T]$ for some $T > 0$:
\begin{equation*}
  k(x, x') = \min(x, x').
\end{equation*}
The resulting kernel $k_\sigma(x,x') = \sigma^2 k(x,x')$ induces a Brownian motion prior for GP regression.

We assume that the input points $\bx=[x_1, \dots x_N]^\top$ for GP regression are positive and ordered:
\begin{equation*}
    0 < x_1 < x_2 < \dots < x_N \leq T.
\end{equation*}
The positivity ensures that the Gram matrix \eqref{eq:gram-matrix} is non-singular; the proof is given in~\Cref{sec:explicit-post-mean-cov}.

As is well-known~\citep[see, for instance,][Example~1]{Diaconis1988} and can be seen in Figures \ref{fig:FBM-H02} and \ref{fig:IFBM-H05}, the posterior mean function $m_N$ in \eqref{eq:posterior-moments} using the Brownian motion kernel becomes the {\em piecewise linear interpolant} of the observations $( \bx, f(\bx) )$. See \eqref{eq:explicit-post-mean} and \eqref{eq:explicit-post-cov} in Section \ref{sec:explicit-post-mean-cov} for the proof and explicit expressions of the posterior mean and covariance functions.

\subsection{Sequences of partitions}
\label{sec:sequence-of-partitions}

For our asymptotic analysis, we assume that the input points $x_1, \dots, x_N \in [0,T]$ cover the domain $[0,T]$ more densely as the sample size $N$ increases. To make the dependence on the size $N$ explicit, we write $\mathcal{P}_N \coloneqq (x_{N,n})_{n=1}^N \subset [0,T]$ as a point set of size $N$, and assume that they are ordered as
\begin{equation*}
    0 \eqqcolon x_{N,0} < x_{N, 1} < x_{N, 2} < \dots < x_{N, N} = T
\end{equation*}%
Then $\mathcal{P}_N$ defines a partition of $[0,T]$ into $N$ subintervals $[x_{N,n}, x_{N, n+1}]$.
When there is no risk of confusion, we may write $x_n$ instead of $x_{N,n}$ for simplicity.
Note that we do {\em not} require the nesting $\mathcal{P}_N \subset \mathcal{P}_{N+1}$ of partitions.

We define the {\em mesh size} of partition $\mathcal{P}_N$ as the longest subinterval in the partition:
$$
\|\prt_N\| \coloneqq \max_{n \in \{0, 1,\ldots,N-1\}} (x_{N, n+1} - x_{N, n} )
$$
The decay rate of the mesh size $\|\prt_N\|$ quantifies how quickly the points in $\mathcal{P}_N$ cover the interval $[0,T]$. In particular, the decay rate $\mathcal{P}_N = \bigo(N^{-1})$ implies that the length of every subinterval is asymptotically upper bounded by $1/N$.
At the same time, if each subinterval is asymptotically lower bounded by $1/N$, we call the sequence of partitions $(\prt_N)_{N \in \mathbb{N}}$ {\em quasi-uniform}, more formally defined in~\citet[Definition 4.6]{Wendland2005} as follows.

\begin{definition}
For each $N \in \mathbb{N}$, let $\mathcal{P}_N \coloneqq (x_{N,n})_{n=1}^N \subset [0,T]$. Define $\Delta x_{N, n} \coloneqq x_{N, n+1} - x_{N, n}$.
Then the sequence of partitions $(\prt_N)_{N \in \mathbb{N}}$ is called {\em quasi-uniform} if there exists a constant  $1 \leq C_\textup{qu} < \infty$ such that
\begin{equation*}
\sup_{N \in \mathbb{N}} \frac{\max_n \Delta x_{N, n}}{\min_n \Delta x_{N, n}} = C_\textup{qu}.
\end{equation*}
\end{definition}

Quasi-uniformity, as defined here, requires that the ratio of the longest subinterval, $\max_{n} \Delta x_{N, n}$, to the shortest one, $\min_{n} \Delta x_{N, n}$, is upper-bounded by $C_{\rm qu}$ for all $N \in \mathbb{N}$. Since $\min_n \Delta x_{N, n} \leq T N^{-1}$ and $\max_n \Delta x_{N, n} \geq T N^{-1}$ for any partition of $[0, T]$, quasi-uniformity implies that all subintervals are asymptotically upper and lower bounded by $1/N$, as we have, for all $N \in \mathbb{N}$ and $n_0 \in \{0, \dots, N-1\}$,
\begin{equation}
  \label{eq:quasi-uniformity-2}
 \frac{T N^{-1} }{C_\textup{qu}}  \leq \min_n \Delta x_{N, n} \leq \Delta x_{N, n_0} \leq \max_n \Delta x_{N, n} \leq T C_\textup{qu} N^{-1}.
\end{equation}
Therefore, quasi-uniform sequences of partitions are \emph{space-filling designs} that cover the space ``almost'' uniformly. Trivially, equally-spaced points (or uniform grids) satisfy the quasi-uniformity with $C_{\rm qu} = 1$.~\citet{wenzel2021novel} showed that points chosen sequentially to minimise GP posterior variance for a Sobolev kernel are quasi-uniform. We refer to~\citet[p.\@~6]{wynne2021convergence} for further examples and a discussion on quasi-uniformity.

\subsection{H\"older spaces} \label{sec:holder-space}

\Cref{sec:results-deterministic} studies the deterministic setting where the true unknown function $f$ is assumed to belong to a H\"older space of functions. To define this space, we first need the following definition.

\begin{definition}
\label{def:holder-continuity}
For $0 < \alpha \leq 1$, a function $f: [0, T] \to \R$ is \textit{$\alpha$-Hölder continuous} if there exists a constant $L \geq 0$ such that, for all $x, x' \in [0, T]$,
\begin{equation*}
    |f(x) - f(x')| \leq L |x - x' |^ \alpha.
\end{equation*}
Any such constant $L$ is called a \emph{H\"older constant} of $f$.
\end{definition}

For $l \in \mathbb{N} \cup \{ 0 \}$, denote by $C^l([0,T])$ the space of functions $f \colon [0, T] \to \R$ such that the $l$\textsuperscript{th} derivative $f^{(l)}$ exists and is continuous. For $l = 0$, this is the space of continuous functions.   H\"older spaces are now defined as follows.

\begin{definition}
  \label{def:hoelder-space}
  Let $l \in \mathbb{N} \cup \{0\}$ and $0 < \alpha \leq 1$. The \textit{H\"older space $C^{l, \alpha}([0, T])$} consists of functions $f \in C^l([0 ,T])$ whose $l$\textsuperscript{th} derivative $f^{(l)}$ is $\alpha$-H\"older continuous.
\end{definition}

Intuitively, $l + \alpha$ represents the smoothness of least-smooth functions in $C^{l, \alpha} ( [0, T] )$. It is well-known that a sample path of Brownian motion is almost surely $\alpha$-H\"older continuous if and only if $\alpha < 1/2$~\citep[e.g.,][Corollary 1.20]{morters2010brownian}, and thus it belongs to the H\"older space $C^{l, \alpha} ([0,T])$ with $l = 0$ and $\alpha = 1/2 - \varepsilon$ almost surely for arbitrarily small $\varepsilon > 0$; in this sense, the smoothness of a Brownian motion is $1/2$. As such, as is well-known~\citep[e.g.,][Theorem 1.27]{morters2010brownian}, a Brownian motion is almost nowhere differentiable almost surely.

Note that we have the following strict inclusions:\footnote{These inclusions follow from the following facts: By the definition of H\"older continuity, an $\alpha_1$-H\"older continuous function is $\alpha_2$-H\"older continuous if $\alpha_1 > \alpha_2$; continuously differentiable functions are $\alpha$-H\"older continuous for any $0< \alpha \leq 1$; not all Lipschitz functions are differentiable.}
\begin{itemize}
  \item $C^{l_1, \alpha_1}([0, T]) \subsetneq C^{l_2, \alpha_2}([0, T])$ if (a) $l_1 > l_2$ or (b) $l_1 = l_2$ and $\alpha_1 > \alpha_2$,
  \item $C^{l+1}([0, T]) \subsetneq C^{l, 1}([0, T])$.
\end{itemize}

\subsection{Fractional Brownian motion} \label{sec:fbm}

\Cref{sec:random-setting} considers the random setting where $f$ is a {\em fractional (or integrated fractional) Brownian motion} (see \citet[e.g.,][Chapter IX]{mandelbrot1982fractal}).
Examples of these processes can be seen in Figures \ref{fig:FBM-H02}, \ref{fig:IFBM-H05}, \ref{fig:cv} and \ref{fig:cv-vs-ml}.

A fractional Brownian motion on $[0,T]$ with Hurst parameter $0 < H  < 1$ is a Gaussian process whose kernel is given by
\begin{equation}
\label{eq:fbm-def}
    k_{0,H}(x,x') =  \big( \, \lvert x \rvert^{2H} + \lvert x' \rvert^{2H} - \lvert x-x'\rvert^{2H}\big) / 2.
\end{equation}
Note that if $H = 1/2$, this is the Brownian motion kernel: $k_{0,1/2}(x,x') = \min (x,x')$.
The Hurst parameter $H$ quantifies the smoothness of the fractional Brownian motion.
If $f_{\rm FBM} \sim \GP(0, k_{0,H})$ for $H \in (0, 1)$, then $f_{\rm FBM} \in C^{0, H-\varepsilon} ( [0, T] )$ almost surely for arbitrarily small $\varepsilon > 0$ ~\citep[e.g.,][Proposition~1.6]{Nourdin2012}.\footnote{That $f_{\rm FBM} \notin C^{0, H} ( [0, T] )$ almost surely for $f_{\rm FBM} \sim \GP(0, k_{0,H})$ with $H \in (0, 1)$ is a straightforward corollary of, for example, Theorem~3.2 in~\citet{wang_almost-sure_2007}.}

An integrated Brownian motion with Hurst parameter $H$ is defined via the integration of a fractional Brownian motion with the same Hurst parameter: if $f_{\rm FBM} \sim \GP(0, k_{0, H})$, then
\begin{equation*}
  f_\text{iFBM}(x) = \int_0^x f_\text{FBM}(z) \, \mathrm{d} z, \quad x \in [0,T]
\end{equation*}
is an integrated Brownian motion with Hurst parameter $H$. It is a zero-mean GP with the kernel
\begin{equation} \label{eq:iFBM-kernel-explicit}
  \begin{split}
  k_{1,H}(x, x') ={}& \int_0^x \int_0^{x'} \big( \lvert z \rvert^{2H} + \lvert z' \rvert^{2H} - \lvert z-z'\rvert^{2H}\big) / 2 \, \mathrm{d} z \, \mathrm{d} z' \\
  ={}& \frac{1}{2(2H+1)} \bigg( x' x^{2H+1} + x (x')^{2H+1} \\
  &\hspace{2cm} - \frac{1}{2(H+1)} \big[ x^{2H+2} + (x')^{2H+2} - |x - x'|^{2H+2} \big] \bigg).
  \end{split}
\end{equation}
Because differentiating an integrated fractional Brownian motion $f_{\rm iFBM} \sim \GP(0, k_{1, H})$ yields a fractional Brownian motion $f_{\rm FBM} \sim \GP(0, k_{0,H})$, a sample path of the former satisfies $f_{\rm iFBM} \in C^{1,H-\varepsilon}([0, T])$ almost surely for arbitrarily small $\varepsilon > 0$; therefore the smoothness of $f_{\rm iFBM}$ is $1 + H$.

\subsection{Functions of finite quadratic variation}
\label{sec:quadratic-variation}

Some of our asymptotic results use the notion of functions of {\em finite quadratic variation}, defined below.
\begin{definition}
For each $N \in \mathbb{N}$, let $\mathcal{P}_N \coloneqq (x_{N,n})_{n=1}^N \subset [0,T]$, and suppose that $\|\prt_N\| \to 0$ as $N \to \infty$.
Then a function $f : [0, T] \to \mathbb{R}$ is defined to have {\em finite quadratic variation} with respect to $\prt \coloneqq (\prt_N)_{N \in \mathbb{N}}$, if the limit
\begin{equation}
\label{eq:quadratic-variation}
    V^2(f) \coloneqq \lim_{N \to \infty } \sum_{n=0}^{N-1} \big[ f(x_{N, n+1}) - f(x_{N, n}) \big]^2
\end{equation}
exists and is finite.
We write $V^2(f, \prt)$ when it is necessary to indicate the sequence of partitions.
\end{definition}

Quadratic variation is defined for a specific sequence of partitions $(\prt_N)_{N \in \mathbb{N}}$ and may take different values for different sequences of partitions~\citep[Remark 1.36]{morters2010brownian}.
For conditions that guarantee the invariance of quadratic variation on the sequence of partitions, see, for instance,~\cite{ContBas2023}.
Note also that the notion of quadratic variation differs from that of $p$-variation for $p=2$, which is defined as the supremum over all possible sequences of partitions whose mesh sizes tend to zero.

If $f \in C^{0,\alpha}([0, T])$ with $\alpha > 1/2$ and $\|\prt_N\| = \bigo(N^{-1})$ as $N \to \infty$, then we have $V^2(f) = 0$, because in this case
\begin{align*}
  \sum_{n=0}^{N-1} \big[ f(x_{N, n+1}) - f(x_{N, n}) \big]^2 \leq   N L^2 \max_n (\Delta x_{N, n})^{2 \alpha} &= \bigo (N^{1 - 2\alpha}) \to 0
\end{align*}
as $N \to \infty$.
Therefore, given the inclusion properties of H\"older spaces (see Section \ref{sec:holder-space}), we arrive at the following standard proposition.

\begin{proposition}
\label{res:qv_of_smooth_functions}
Suppose that the partitions $(\prt_N)_{N \in \mathbb{N}}$ are such that $\| \prt_N \| = \bigo (N^{-1})$. If $f \in C^{l, \alpha}([0, T])$ for $l+\alpha>1/2$, then $V^2(f) = 0$.
\end{proposition}
If the mesh size tends to zero faster than $1/\log N$, in that $\| \prt_N \|=o(1/\log N)$, then the quadratic variation of almost every sample path of the Brownian motion on the interval $[0, T]$ equals~$T$~\citep{10.2307/2959347}.
This is of course true for partitions which have the faster decay $\| \prt_N \|=\bigo(N^{-1})$.

\section{Main results}
\label{sec:limit-behaviour-for-sigma}

This section presents our main results on the asymptotic properties of the CV and ML estimators, $\sigmasqcvest$ and $\sigmasqmlest$, for the scale parameter. \Cref{sec:results-deterministic} considers the deterministic setting where the true function $f$ is fixed and assumed to belong to a H\"older space. \Cref{sec:random-setting} studies the random setting where $f$ is an (integrated) fractional Brownian motion. In~\Cref{sec:icv-estimators}, we use the insights obtained in the proofs for the deterministic and random settings to propose a \emph{interior cross-validation} (ICV) estimator, and show its asymptotic properties are an improvement on those of CV and ML estimators.

\subsection{Deterministic setting} \label{sec:results-deterministic}

We present our main results for the deterministic case where the true function $f$ is fixed and assumed to be in a H\"older space $C^{l, \alpha}([0, T])$.  \Cref{res:holder-spaces} below provides asymptotic upper bounds on the CV estimator $\sigmasqcvest$ for different values of the smoothness parameters $l$ and $\alpha$ of the H\"older space.

\begin{theorem}[Rate of CV decay in H\"older spaces]
\label{res:holder-spaces}
Suppose that $f$ is an element of $C^{l, \alpha}([0, T])$, with $l \geq 0$ and $0 < \alpha \leq 1$, such that  $f(0)=0$, and the interval partitions $(\prt_N)_{N \in \mathbb{N}}$ have bounded mesh sizes $\|\prt_N \|=\bigo(N^{-1})$ as $N \to \infty$. Then
\begin{equation}
\label{eq:main-result}
\sigmasqcvest = \bigo\big( N^{1 - \min\{2(l + \alpha), 3\}} \big)
= \begin{cases}
    \bigo\left(N^{1 - 2 \alpha}\right) &\text{ if } \quad l = 0, \\
    \bigo\left(N^{-1 - 2\alpha}\right) &\text{ if } \quad  l = 1 \text{ and } \alpha < 1/2, \\
    \bigo\left(N^{- 2}\right) &\text{ if } \quad l = 1 \text{ and } \alpha \geq 1/2, \\
    \bigo\left(N^{- 2}\right) &\text{ if } \quad l \geq 2.
    \end{cases}
\end{equation}
\end{theorem}
\begin{proof}
See Section~\ref{sec:proofs-deterministic}.
\end{proof}

\Cref{res:holder-spaces-ml} below is a corresponding result for the ML estimator $\sigmasqmlest$. Note that a similar result has been obtained by \citet[Proposition~4.5]{Karvonen2020}, where the function $f$ is assumed to belong to a Sobolev space and the kernel is a Mat\'ern-type kernel. \Cref{res:holder-spaces-ml} is a version of this result where $f$ is in a H\"older space and the kernel is the Brownian motion kernel; we provide it for completeness and ease of comparison.

\begin{theorem}[Rate of ML decay in H\"older spaces]
\label{res:holder-spaces-ml}
Suppose that $f$ is a non-zero element of $C^{l, \alpha}([0, T])$, with $l \geq 0$ and $0 < \alpha \leq 1$, such that  $f(0)=0$, and the interval partitions $(\prt_N)_{N \in \mathbb{N}}$ have bounded mesh sizes $\|\prt_N \|=\bigo(N^{-1})$ as $N \to \infty$. Then
\begin{equation}
\label{eq:main-result-ml}
\hat{\sigma}_\textup{ML}^2 = \bigo\big( N^{1 - \min\{2(l + \alpha), 2\}} \big)
= \begin{cases}
    \bigo\left(N^{1 - 2 \alpha}\right) &\text{ if } \quad l = 0, \\
    \Theta\left(N^{- 1}\right) &\text{ if } \quad l \geq 1.
    \end{cases}
\end{equation}
\end{theorem}
\begin{proof}
  See Section~\ref{sec:proofs-deterministic}. The proof is similar to that of \Cref{res:holder-spaces}.
\end{proof}

Figure \ref{fig:rate-regimes-intro-deterministic} summarises the rates of \Cref{res:holder-spaces,res:holder-spaces-ml}. When $l + \alpha \leq 1$ (or $l = 0$ and $\alpha \leq 1$), the rates of  $\sigmasqcvest$ and $\sigmasqmlest$ are $\bigo(N^{1-2\alpha})$, so both of them may decay (or grow, for $l+\alpha<1/2$) adaptively to the smoothness $l+\alpha$ of the function $f$. However, when $l + \alpha > 1$, the situation is different: the decay rate of $\sigmasqmlest$ is always $\Theta(N^{-1})$ and thus insensitive to $\alpha$, while that of $\sigmasqcvest$ is $\left(N^{-1 - 2\alpha}\right)$ for $l=1$ and $\alpha \in (0, 1/2]$. Therefore the CV estimator may be adaptive to a broader range of the smoothness $0 < l + \alpha \leq 3/2$ of the function $f$ than the ML estimator (whose range of adaptation is $0 < l + \alpha \leq 1$).

Note that \Cref{res:holder-spaces,res:holder-spaces-ml} provide asymptotic upper bounds (except for the case $l \geq 1$ of \Cref{res:holder-spaces-ml}) and may not be tight if the function $f$ is smoother than ``typical'' functions in $C^{l,\alpha}([0,T])$.\footnote{For example, if $f(x) = | x - 1/2 |$ with $T = 1$, we have $f \in C^{0,1} ([0,T])$, as $f$ is Lipschitz continuous in this case. However, $f$ is almost everywhere infinitely differentiable except at one point $x = 1/2$, so it is, in this sense, much smoother than ``typical'' functions in $C^{0,1} ([0,T])$.} In \Cref{sec:random-setting}, we show that the bounds are indeed tight in expectation if $f$ is a fractional (or integrated fractional) Brownian motion with smoothness $l + \alpha$.

In the deterministic setting, a potential approach for obtaining a matching lower bound could use the rate of decay of the Fourier coefficients as a notion of smoothness, instead of the H\"older smoothness condition on the function $f$. Certain self-similarity conditions based on the decay rate and behaviour of Fourier coefficients are routinely used to study coverage of Bayesian credible sets~\citep[e.g.,][]{Szabo2015, HadjiSzabo2021} as they define classes of functions that cannot ``deceive'' parameter estimators. Motivated by this, we attempted to adapt the argument in~\citet[Section~4.2]{sniekers2015adaptive} and \citet[Section~10]{sniekers2020adaptive} to derive a matching lower bound under a self-similarity assumption on the Fourier coefficients. However, the bounds obtained through this approach proved sub-optimal in our setting. A different technique may therefore be required.

\begin{figure}
    \centering
    \includegraphics[width=0.8\textwidth]{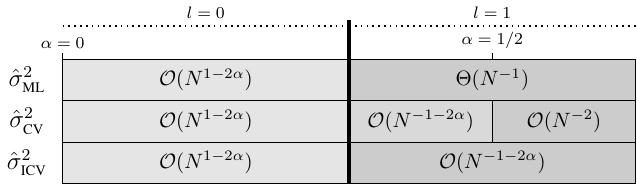}
    \caption{Rates of decay for the ML, CV and ICV estimators from \Cref{res:holder-spaces,res:holder-spaces-ml,res:holder-spaces-icv}.
    Observe that the CV estimator's range of adaptation to the smoothness $l+\alpha$ is wider than the ML estimator's, and the ICV estimator's range of adaptation is wider than that for both the CV and ML estimators.}
    \label{fig:rate-regimes-intro-deterministic}
\end{figure}

\begin{remark}
The proof of \Cref{res:holder-spaces-ml} shows that for $l = 1$ we have $\sigmasqmlest = \Theta(N^{-1})$ whenever $\|\prt_N \| \to 0$ as $N \to \infty$.
More precisely, it establishes that
  \begin{equation*}
    N \sigmasqmlest \to \| f' \|_{\mathcal L^2([0, T])} :=\int_0^T f'(x)^2 \, \mathrm{d} x \quad \text{ as } \quad N \to \infty.
  \end{equation*}
 Note that the $\mathcal L^2([0, T])$ norm of $f'$ in the right hand side equals the norm of $f$ in the reproducing kernel Hilbert space of the Brownian motion kernel~\citep[e.g.,][Section~10]{VaartZanten2008}
Therefore, this fact is consistent with a similar more general statement in \citet[Proposition~3.1]{Karvonen2020}.

\end{remark}

In addition to the above results, \Cref{res:fqv-estimator} below shows the limit of the CV estimator $\sigmasqcvest$ if the true function $f$ is of finite quadratic variation.

\begin{theorem}
\label{res:fqv-estimator}
For each $N \in \mathbb{N}$,  let $\prt_N \subset [0,T]$ be the equally-spaced partition of size $N$.
Suppose that $f: [0,T] \to \mathbb{R}$ has finite quadratic variation $V^2(f)$ with respect to $(\prt_N)_{N \in \mathbb{N}}$, $f(0) = 0$, and $f$ is continuous on the boundary, i.e., $\lim_{x \to 0^+} f(x) = f(0)$ and $\lim_{x \to T^-} f(x) = f(T)$.
Moreover, suppose that the quadratic variation $V^2(f)$ remains the same for all sequences of quasi-uniform partitions with constant $C_\mathrm{qu}=2$.\footnote{In \Cref{sec:discussion-thm-fqv-estimator}, we discuss the relaxation of this requirement.}
Then
\begin{equation} \label{eq:sigma-cv-limit-qvar}
  \lim_{N \to \infty} \sigmasqcvest = \frac{V^2(f)}{T}.
\end{equation}
\end{theorem}
\begin{proof}
See \Cref{sec:proofs-deterministic}.
\end{proof}

For the ML estimator $\sigmasqmlest$, it is straightforward to obtain a similar result by using \eqref{eq:quasi-uniformity-2} and \eqref{eq:sigma-ml-w-deltas} in Section \ref{sec:explicit-post-mean-cov}: Under the same conditions as \Cref{res:fqv-estimator}, we have
\begin{equation} \label{eq:ml-quad-var}
  \lim_{N \to \infty} \sigmasqmlest =\frac{V^2(f)}{T}.
\end{equation}

\Cref{res:fqv-estimator} and \eqref{eq:ml-quad-var} are consistent with \Cref{res:holder-spaces,res:holder-spaces-ml}, which assume $f \in C^{l, \alpha}([0,T])$ with $l + \alpha > 1/2$ and imply $\sigmasqcvest \to 0$ and $\sigmasqmlest \to 0$ as $N \to \infty$. As summarised in~\Cref{res:qv_of_smooth_functions}, we have $V(f) = 0$ for $f \in C^{l, \alpha}([0,T])$ with $l + \alpha > 1/2$, so \Cref{res:fqv-estimator} and \eqref{eq:ml-quad-var} imply that $\sigmasqcvest \to 0$ and $\sigmasqmlest \to 0$  as $N \to \infty$.

When $f$ is a Brownian motion, in which case the Brownian motion prior is well-specified, the smoothness of $f$ is $l + \alpha = 1/2$, and the quadratic variation $V(f)$ becomes a positive constant ~\citep{10.2307/2959347}. \Cref{prop:bm-almost-sure} in the next subsection shows that this fact,  \Cref{res:fqv-estimator}, and \eqref{eq:ml-quad-var} lead to the consistency of the ML and CV estimators in the well-specified setting.

\subsection{Random setting} \label{sec:random-setting}

In \Cref{sec:results-deterministic}, we obtained asymptotic upper bounds on the CV and ML scale estimators when the true function $f$ is a fixed function in a H\"older space. This section shows that these asymptotic bounds are tight in expectation when $f$ is a fractional (or integrated fractional) Brownian motion.

That is, we consider the asymptotics of the expectations $\E \sigmasqcvest$ and $\E \sigmasqmlest$ under the assumption that $f \sim \GP(0, k_{l, H})$, where $k_{l, H}$ is the kernel of a fractional Brownian motion \eqref{eq:fbm-def} for $l = 0$ or that of an integrated fractional Brownian motion \eqref{eq:iFBM-kernel-explicit} for $l = 1$, with $0 < H <1$ being the Hurst parameter. Recall that $f \sim \GP(0, k_{l, H})$ belongs to the H\"older space $C^{l, H - \varepsilon}([0,T])$ almost surely for arbitrarily small $\varepsilon > 0$, so its smoothness is $l + H$.  Figure \ref{fig:rate-regimes-intro} summarises the obtained upper and lower rates, corroborating the upper rates in Figure \ref{fig:rate-regimes-intro-deterministic}.

\Cref{res:holder-spaces-exp,res:holder-spaces-exp-ml} below establish the asymptotic upper and lower bounds for the CV and ML estimators, respectively.

\begin{theorem}[Expected CV rate for fractional Brownian motion]
\label{res:holder-spaces-exp}
    Suppose that $(\prt_N)_{N \in \mathbb{N}}$ are quasi-uniform and $f \sim \GP(0, k_{l,H})$ with $l \in \{0, 1\}$ and $0 < H < 1$.
    Then
    \begin{equation*}
      \E \sigmasqcvest = \Theta ( N^{1 - \min\{2(l + H),3\}} ) =
      \begin{cases}
        \Theta\left(N^{1 - 2  H}\right) &\text{ if } \quad l = 0 \text{ and } H \in (0, 1), \\
        \Theta\left(N^{-1 - 2H}\right) &\text{ if } \quad  l = 1 \text{ and } H < 1/2, \\
        \Theta\left(N^{- 2}\right) &\text{ if } \quad l = 1 \text{ and } H \geq 1/2. \\
      \end{cases}
    \end{equation*}
\end{theorem}
\begin{proof}
  See Section~\ref{sec:proofs-random}.
\end{proof}

\begin{theorem}[Expected ML rate for fractional Brownian motion]
\label{res:holder-spaces-exp-ml}
      Suppose that $(\prt_N)_{N \in \mathbb{N}}$ are quasi-uniform and $f \sim \GP(0, k_{l,H})$ with $l \in \{0, 1\}$ and $0 < H < 1$.
    Then
    \begin{equation*}
      \E \hat{\sigma}_\textup{ML}^2 = \Theta ( N^{1 - \min\{2(l + H),2\}} ) =
      \begin{cases}
        \Theta\left(N^{1 - 2  H}\right) &\text{ if } \quad l = 0 \text{ and } H \in (0, 1), \\
        \Theta\left(N^{-1}\right) &\text{ if } \quad  l = 1 \text{ and } H \in (0, 1).
      \end{cases}
    \end{equation*}
\end{theorem}
\begin{proof}
  See Section~\ref{sec:proofs-random}. The proof is similar to that of \Cref{res:holder-spaces-exp}.
\end{proof}

\Cref{res:holder-spaces-exp,res:holder-spaces-exp-ml} show that the CV estimator is adaptive to the unknown smoothness $l + H$ of the function $f$ for a broader range $0< l+H \leq 3/2$ than the ML estimator, whose range of adaptation is $0 < l+H \leq 1$. These results imply that the CV estimator can be asymptotically well-calibrated for a broader range of unknown smoothness than the ML estimator, as discussed in \Cref{sec:discussion}.

When the smoothness of $f$ is less than $1/2$, i.e., when $l + H < 1/2$, the Brownian motion prior, whose smoothness is $1/2$, is smoother than $f$. In this case, the expected rates of $\sigmasqcvest$ and $\sigmasqmlest$ are $ \Theta\left(N^{1 - 2  H}\right)$ and increase as $N$ increases. The increase of $\sigmasqcvest$ and $\sigmasqmlest$ can be interpreted as compensating the overconfidence of the posterior standard deviation $\sqrt{\smash[b]{k_N(x)}}$, which decays too fast to be asymptotically well-calibrated. This interpretation agrees with the illustration in Figure \ref{fig:FBM-H02}.

On the other hand, when $l+ H > 1/2$, the function $f$ is smoother than the Brownian motion prior. In this case,  $ \sigmasqcvest$ and $\sigmasqmlest$ decrease as $N$ increases, compensating the under-confidence of the posterior standard deviation $\sqrt{\smash[b]{k_N(x)}}$. See Figure \ref{fig:IFBM-H05} for an illustration.

When $l + H = 1/2$, this is the well-specified case in that the smoothness of $f$ matches the Brownian motion prior. In this case,  \Cref{res:holder-spaces-exp,res:holder-spaces-exp-ml} yield $\E \sigmasqcvest = \Theta(1)$ and $\E \sigmasqmlest = \Theta(1)$, i.e., when the CV and ML estimators converge, they converge to a positive constant.
The following proposition, which follows from \Cref{res:fqv-estimator} and \eqref{eq:ml-quad-var}, shows that this limiting constant is the true value of the scale parameter $\sigma_0^2$ in the well-specified setting $f \sim \GP(0, \sigma_0^2 k)$, recovering similar results in the literature \citep[e.g.,][Theorem 2]{Bachoc2017}.

\begin{proposition} \label{prop:bm-almost-sure}
  Suppose that $f \sim \GP(0, \sigma_0^2 k)$ for $\sigma_0 > 0$ and that partitions  $(\prt_N)_{N \in \mathbb{N}}$ are equally-spaced. Then
  \begin{equation*}
    \lim_{N \to \infty} \sigmasqcvest = \lim_{N \to \infty} \sigmasqmlest = \sigma_0^2 \quad \text{ almost surely}.
  \end{equation*}
\end{proposition}
\begin{proof}
  Since the quadratic variation of almost all sample paths of the unscaled (i.e., $\sigma_0 = 1$) Brownian motion on $[0, T]$ equals $T$~\citep{10.2307/2959347}, the claim follows from~\eqref{eq:sigma-cv-limit-qvar} and~\eqref{eq:ml-quad-var}.
\end{proof}

In~\Cref{sec:discussion}, we discuss the implications of the obtained asymptotic rates of $\sigmasqcvest$ and $\sigmasqmlest$ on the reliability of the resulting GP uncertainty estimates. But first, motivated by the results in~\Cref{res:holder-spaces} and~\Cref{res:holder-spaces-exp}, we propose a modification to the cross-validation procedure that may have better asymptotic properties than the CV estimator.

\begin{figure}
    \centering
    \includegraphics[width=0.8\textwidth]{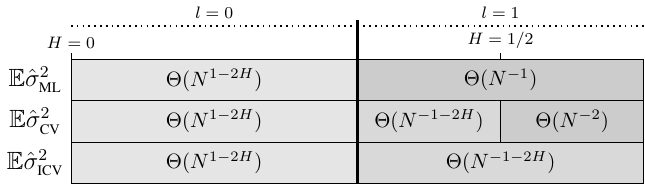}
    \caption{Expected decay rates for the ML, CV and ICV estimators from \Cref{res:holder-spaces-exp,res:holder-spaces-exp-ml,res:holder-spaces-icv-exp}.
    Observe that the CV estimator's range of adaptation to the smoothness $l+H$ is wider than the ML estimator's, and the ICV estimator's range of adaptation is wider than that for both the CV and ML estimators.
    %Observe that the CV estimator has a wider range of possible decay rates.
    }
    \label{fig:rate-regimes-intro}
\end{figure}

\subsection{Interior cross-validation estimators}
\label{sec:icv-estimators}

The proofs of~\Cref{res:holder-spaces,res:holder-spaces-exp} show that when $l=1$ and $\alpha \in (1/2, 1]$, the bound on $\sigmasqcvest$ is dominated by the bound on what we call the boundary terms.
These are the terms corresponding to $n = 1$ and $n = N$ in~\eqref{eq:sigma-cv}; see also~\eqref{eq:boundary-terms}.
That the boundary terms dominate is unsurprising since prediction at boundary points is a more challenging task than prediction at the interior. Motivated by this observation, we propose an alternative estimation method called \emph{interior cross validation} (ICV) that maximises
\begin{equation*}
\label{eq:icv-estimator-any-kernel}
  \sum_{n=2}^{N-1} \log p( f(x_n) \given x_n, \bx_{\setminus n}, f(\bx_{\setminus n}), \theta) .
\end{equation*}
The corresponding scale parameter estimator is
\begin{equation}
\label{eq:icv-estimator}
    \sigmasqicvest = \frac{1}{N} \sum_{n=2}^{N-1}  \frac{\left[f(x_n) - m_{\setminus  n}(x_n)\right]^2}{ k_{\setminus  n}(x_n)}.
\end{equation}
With the boundary points removed, the estimator's range of adaptation to the smoothness of the true function is greater than that for the CV estimator, as illustrated in~\Cref{fig:rate-regimes-intro-deterministic} for the deterministic setting and~\Cref{fig:rate-regimes-intro} for the random setting. We present formal results for the deterministic and the random settings in the following theorems.

\begin{theorem}[Rate of ICV decay in H\"older spaces]
\label{res:holder-spaces-icv}
Suppose that $f$ is an element of $C^{l, \alpha}([0, T])$, with $l \geq 0$ and $0 < \alpha \leq 1$, such that $f(0)=0$, and the interval partitions $(\prt_N)_{N \in \mathbb{N}}$ have bounded mesh sizes $\|\prt_N \|=\bigo(N^{-1})$ as $N \to \infty$. Then
\begin{equation*}
\sigmasqicvest = \bigo\big( N^{1 - \min\{2(l + \alpha), 4\}} \big)
= \begin{cases}
    \bigo\left(N^{1 - 2 \alpha}\right) &\text{ if } \quad l = 0, \\
    \bigo\left(N^{-1 - 2\alpha}\right) &\text{ if } \quad  l = 1, \\
    \bigo\left(N^{- 3}\right) &\text{ if } \quad l \geq 2.
    \end{cases}
\end{equation*}
\end{theorem}
\begin{proof}
See Section~\ref{sec:proofs-icv}.
\end{proof}

\begin{theorem}[Expected ICV rate for fractional Brownian motion]
\label{res:holder-spaces-icv-exp}
    Suppose that $(\prt_N)_{N \in \mathbb{N}}$ are quasi-uniform and $f \sim \GP(0, k_{l,H})$ with $l \in \{0, 1\}$ and $0 < H < 1$.
    Then
    \begin{equation*}
      \E \sigmasqicvest = \Theta ( N^{1 - \min\{2(l + H),4\}} ) =
      \begin{cases}
        \Theta\left(N^{1 - 2  H}\right) &\text{ if } \quad l = 0, \\
        \Theta\left(N^{-1 - 2H}\right) &\text{ if } \quad  l = 1.
        % \Theta\left(N^{- 3}\right) &\text{ if } \quad l \geq 2. \\
      \end{cases}
    \end{equation*}
\end{theorem}
\begin{proof}
  See Section~\ref{sec:proofs-icv}.
\end{proof}

This idea can be taken further. For the Brownian motion kernel, an estimator that does not attempt to predict on points ``close enough'' to the boundary,
\begin{equation*}
    \sigmasqicvest[N_0] = \frac{1}{N} \sum_{n=N_0}^{N-N_0}  \frac{\left[f(x_n) - m_{\setminus  n}(x_n)\right]^2}{ k_{\setminus  n}(x_n)}
\end{equation*}
for some fixed $N_0$, has the same range of adaptation as $\sigmasqicvest= \sigmasqicvest[1]$, the estimator that only ignores the points on the boundary. However, for smoother kernels like fractional Brownian motion (iFBM) and the Mat\'ern family, $\sigmasqicvest[N_0]$ may exhibit adaptation beyond the level $l=2$. The number of boundary points $N_0$ to remove would likely depend on the smoothness of the kernel. Investigating model-dependent cross-validation estimators that discard a proportion of boundary points would be an interesting direction for future work.

\section{Consequences for credible intervals}
\label{sec:discussion}

This section discusses whether the estimated scale parameter, given by the CV or ML estimator, leads to asymptotically well-calibrated credible intervals.
 With the kernel $\hat{\sigma}^2 k(x,x')$, where $\hat{\sigma}^2 = \sigmasqcvest$ or   $\hat{\sigma}^2 = \sigmasqmlest$,  a GP credible interval at $x \in [0,T]$ is given by
 \begin{equation} \label{eq:CI-discus}
     [m_N(x) - \alpha \hat{\sigma} \sqrt{k_N(x)},\quad m_N(x) + \alpha \hat{\sigma} \sqrt{k_N(x)}]
 \end{equation}
where $\alpha > 0$ is a constant (e.g., $\alpha \approx 1.96$ leads to the 95\% credible interval).

As discussed in Section \ref{sec:introduction}, this credible interval \eqref{eq:CI-discus} is asymptotically well-calibrated, if it shrinks to $0$ at the same speed as the decay of the error $|m_N(x) - f(x)|$ as $N$ increases, i.e., the ratio
\begin{equation} \label{eq:ratio-discus}
    \frac{|f(x) - m_N(x)| }{ \hat{\sigma} \sqrt{k_N(x)}}
\end{equation}
should neither diverge to infinity nor converge to $0$.
If this ratio diverges to infinity, the credible interval \eqref{eq:CI-discus} is asymptotically overconfident, in that \eqref{eq:CI-discus} shrinks to $0$ faster than the actual error $|f(x) - m_N(x)|$. If the ratio converges to $0$, the credible interval is asymptotically underconfident, as it increasingly overestimates the actual error. Therefore, the ratio \eqref{eq:ratio-discus} should ideally converge to a positive constant for the credible interval \eqref{eq:CI-discus} to be reliable.

For ease of analysis,  we focus on the random setting in Section \ref{sec:random-setting} where $f$ is a fractional (or integrated fractional) Brownian motion and where we obtained asymptotic upper and lower bounds for $\E \sigmasqcvest$ and $\E \sigmasqmlest$.
We study how the expectation of the posterior variance $\E \hat{\sigma}^2 k_N(x)$ scales with the expected squared error $\E[ f(x) - m_N(x) ]^2$.
Specifically, we analyse their ratio for $\hat{\sigma}^2 = \sigmasqcvest$ and $\hat{\sigma}^2 = \sigmasqmlest$:
\begin{equation} \label{eq:ratios-expectation}
  R_\textup{CV}^\E(x, N) \coloneqq \frac{\E[ f(x) - m_N(x) ]^2}{\E \hat{\sigma}_\textup{CV}^2 k_N(x) } \quad \text{ and } \quad R_\textup{ML}^\E(x, N) \coloneqq \frac{\E[ f(x) - m_N(x) ]^2}{\E \hat{\sigma}_\textup{ML}^2 k_N(x) }.
\end{equation}
The ratio diverging to infinity (resp.~converging to $0$) as $N \to \infty$ suggests that the credible interval \eqref{eq:CI-discus} is asymptotically overconfident (resp.~underconfident) for a non-zero probability of the realisation of $f$. Thus ideally, the ratio should converge to a positive constant.

\Cref{res:uq-theorem-exp} below establishes the asymptotic rates of the ratios in \eqref{eq:ratios-expectation}.

\begin{theorem}
  \label{res:uq-theorem-exp}
  Suppose that $(\prt_N)_{N \in \mathbb{N}}$ are quasi-uniform and $f \sim \GP(0, k_{l, H})$ for $l \in \{0,1 \}$ and $0 < H < 1$. Then,
  \begin{equation*}
    \sup_{x \in [0, T]} R_\textup{CV}^\E(x, N)
    =
      \begin{cases}
        \Theta(1) &\text{ if } \quad l = 0 \text{ and } H \in (0, 1), \\
        \Theta(1) &\text{ if } \quad l = 1 \text{ and } H \in (0, 1/2), \\
        \Theta\left(N^{1-2H}\right) &\text{ if } \quad l = 1 \text{ and } H \in (1/2, 1), \\
      \end{cases}
  \end{equation*}
  and
  \begin{equation*}
    \sup_{x \in [0, T]} R_\textup{ML}^\E(x, N)
    =
      \begin{cases}
        \Theta(1) &\text{ if } \quad l = 0 \text{ and } H \in (0, 1), \\
        \Theta\left(N^{-2H}\right) &\text{ if } \quad l = 1 \text{ and } H \in (0, 1).
      \end{cases}
  \end{equation*}
\end{theorem}
\begin{proof}
  See \Cref{sec:proofs-discussion}.
\end{proof}

We have the following observations from \Cref{res:uq-theorem-exp}, which suggest an advantage of the CV estimator over the ML estimator for uncertainty quantification:
\begin{itemize}
    \item The ratio for the CV estimator neither diverges to infinity nor decays to $0$ in the range $0 < l+H < 3/2$, which is broader than that of the ML estimator, $0 < l+H < 1$. This observation suggests that the CV estimator can yield asymptotically well-calibrated credible intervals for a broader range of the unknown smoothness $l + H$ of the function $f$ than the ML estimator.

    \item The ratio decays to $0$ for the CV estimator in the range $3/2 < l+H < 2$ and for the ML estimator in the range $1 < l+H < 2$. Therefore, the ML estimator may yield asymptotically underconfident credible intervals for a broader range of the smoothness $l+H$ than the CV estimator.
\end{itemize}
Moreover, for the interior CV estimator introduced in~\Cref{eq:icv-estimator}, it immediately follows from the proof in~\Cref{sec:proofs-discussion} that
\begin{equation*}
\sup_{x \in [0, T]} R_\textup{ICV}^\E(x, N)
=
    \begin{cases}
    \Theta(1) &\text{ if } \quad l = 0 \text{ and } H \in (0, 1), \\
    \Theta(1) &\text{ if } \quad l = 1 \text{ and } H \in (0, 1),
    \end{cases}
\end{equation*}
which implies the ICV estimator can yield asymptotically well-calibrated credible intervals for a broader range of the smoothness than either the CV or the ML estimator.

\section{Experiments}
\label{sec:experiments}

This section describes numerical experiments to substantiate the theoretical results in \Cref{sec:limit-behaviour-for-sigma}.
We define test functions in \Cref{sec:test-functions}, show empirical asymptotic results for the CV estimator in \Cref{sec:experiment-cv}, and report comparisons between the CV and ML estimators in \Cref{sec:comparison-CV-ML}.

To this end, for a continuous function $f$, define $l[f] \in \mathbb{N} \cup \{ 0 \}$ and $\alpha \in (0, 1]$ as
\begin{equation} \label{eq:smoothness-1037}
        l[f] := \sup\{ l \in \mathbb{N} \cup \{ 0 \} : f \in C^{l}([0, T]) \}, \quad \alpha[f] := \sup\{ \alpha \in (0, 1] : f \in C^{l[f],\alpha}([0, T]) \}.
\end{equation}
Then, for arbitrarily small $\varepsilon_1 \in \mathbb{N}$ and $\varepsilon_2 > 0$, we have
\begin{equation*}
  f \in C^{\max ( l[f]-\varepsilon_1, 0),\alpha[f]-\varepsilon_2}([0, T]) \quad \text{ and } \quad f \notin C^{l[f]+\varepsilon_1,\alpha[f]+\varepsilon_2}([0, T]).
\end{equation*}
In this sense, $l[f]$ and $\alpha[f]$ characterise the smoothness of $f$.

\subsection{Test functions}
\label{sec:test-functions}

We generate test functions $f: [0,1] \to \mathbb{R}$ as sample paths of stochastic processes with varying degrees of smoothness, as defined below.
The left columns of \Cref{fig:cv,fig:cv-vs-ml} show samples of these functions.

\begin{itemize}
    \item
To generate nowhere differentiable test functions, we use the Brownian motion (BM), the Ornstein--Uhlenbeck process (OU), and the fractional Brownian motion (FBM\footnote{We use \url{https://github.com/crflynn/fbm} to sample from FBM.}) which are zero-mean GPs with kernels
\begin{align*}
    k_\textup{BM}(x,x') & = \min(x, x'),\quad   k_\textup{OU}(x,x')  = \big(e^{- \lambda \lvert x-x' \rvert} - e^{-\lambda (x+x')}\big) / 4,  \\
    k_\textup{FBM}(x,x') & =  \big( \, \lvert x \rvert^{2H} + \lvert x' \rvert^{2H} - \lvert x-x'\rvert^{2H}\big) / 2,
\end{align*}
where $\lambda > 0$ and $0<H<1$ is the Hurst parameter (recall that the FBM $=$ BM if $H = 1/2$). We set $\lambda = 0.2$ in the experiments below.
Almost all samples $f$ from these processes satisfy $l[f] = 0$.
For BM and OU we have $\alpha[f] = 1/2$ and for FBM $\alpha[f] = H$ (see \Cref{sec:fbm}).
It is well-known that the OU process with the kernel $k_\textup{OU}$ above satisfies the stochastic differential equation
\begin{equation} \label{eq:OU-SDE}
  \mathrm{d} f(t) = -\lambda f(t) \mathrm{d} t + \sqrt{\frac{\lambda}{2}} \, \mathrm{d} B(t),
\end{equation}
where $B$ is the standard Brownian motion whose kernel is $k_{\rm BM}$.

\item
To generate differentiable test functions, we use once (iFBM) and twice (iiFBM) integrated fractional Brownian motions
\begin{equation*}
    f_\textup{iFBM}(x) =\int_0^x f_\textup{FBM}(z) \, \mathrm{d} z \quad \text{ and } \quad f_\textup{iiFBM}(x) =\int_0^x f_\textup{iFBM}(z) \, \mathrm{d} z,
\end{equation*}
where $f_\textup{FBM} \sim \GP(0, k_{\rm FBM})$.
See~\eqref{eq:iFBM-kernel-explicit} for the iFBM kernel.
With $H$ the Hurst parameter of the original FBM, almost all samples $f$ from the above processes satisfy $l[f] = 1$ and $\alpha[f] = H$ (iFBM) or $l[f] = 2$ and $\alpha[f] = H$ (iiFBM).

\item
We also consider a piecewise infinitely differentiable function $f(x) = \sin 10x + [x>x_0]$, where $x_0$ is randomly sampled from the uniform distribution on $[0,1]$ and $[x > x_0]$ is $1$ if $x > x_0$ and $0$ otherwise. This function is of finite quadratic variation with $V^2(f) = 1$.
\end{itemize}

Denote $\hat{\sigma}^2 = \lim_{N \to \infty} \sigmasqcvest$.
 For the above test functions, with equally-spaced partitions, we expect the following asymptotic behaviours for the CV estimator from
\Cref{res:holder-spaces,res:fqv-estimator,res:holder-spaces-exp}, \Cref{prop:bm-almost-sure}, the definition of quadratic variation, and Equation~\eqref{eq:OU-SDE}:
\begin{align*}
  \text{BM ($l[f]=0$, $\alpha[f]=1/2$):} &\quad\quad \sigmasqcvest = \bigo(1) \quad &&\hspace{-1cm}\text{ and } \quad \hat{\sigma}^2 = 1, \\
  \text{OU ($l[f]=0$, $\alpha[f]=1/2$):} &\quad\quad \sigmasqcvest = \bigo(1) \quad &&\hspace{-1cm}\text{ and } \quad \hat{\sigma}^2 = \lambda/2, \\
  \text{FBM  ($l[f]=0$, $\alpha[f]=H$):} &\quad\quad \sigmasqcvest = \bigo(N^{1 - 2H}) \quad &&\hspace{-1cm}\text{ and } \quad \hat{\sigma}^2 = 0, \\
  \text{iFBM  ($l[f]=1$, $\alpha[f]=H$):} &\quad\quad \sigmasqcvest = \bigo(N^{-1 - 2H}) \quad &&\hspace{-1cm}\text{ and } \quad \hat{\sigma}^2 = 0, \\
  \text{iiFBM  ($l[f]=2$, $\alpha[f]=H$):} &\quad\quad \sigmasqcvest = \bigo(N^{-2}) \quad &&\hspace{-1cm}\text{ and } \quad \hat{\sigma}^2 = 0, \\
  \sin 10x + [x > x_0]: &\quad\quad \sigmasqcvest = \bigo(1) \quad &&\hspace{-1cm}\text{ and } \quad \hat{\sigma}^2 = 1.
\end{align*}
Note that the above rate for the iFBM holds for $0 < H \leq 1/2$.
The chosen functions allow us to cover a range of $\alpha[f]$ and $l[f]$ relevant to the varying rate of convergence in~\Cref{res:holder-spaces,res:holder-spaces-exp}, as well as a range of $V^2(f)$ relevant to the limit in~\Cref{res:fqv-estimator}, $\lim_{N \to \infty} \sigmasqcvest = V^2(f) / T$.

\subsection{Asymptotics of the CV estimator}
\label{sec:experiment-cv}

\Cref{fig:cv} shows the asymptotics of $\sigmasqcvest$, where each row corresponds to one stochastic process generating test functions $f$; the rows are displayed in the increasing order of smoothness as quantified by $l[f] + \alpha[f]$.
The estimates are obtained for equally-spaced partitions of sizes $N=10,10^2,\dots,10^5$.
In each row, the left panel plots a single sample of generated test functions $f$. The middle panel shows the mean and confidence intervals (of two standard deviations) of $\sigmasqcvest$ for 100 sample realisations of $f$ for each sample size $N$. The right panel describes the convergence rate of $\sigmasqcvest$ to its limit point $\hat{\sigma}^2 = \lim_{N \to \infty} \sigmasqcvest$ on the log scale.

We have the following observations:
\begin{itemize}
\item The first two rows (the FBM and OU) and the last (the piece-wise infinitely differentiable function) confirm \Cref{res:fqv-estimator}, which states the convergence $\sigmasqcvest \to V^2(f) / T$ as $N \to \infty$. While \Cref{res:fqv-estimator} does not provide convergence rates, the rates in the first two rows appear to be $N^{-1/2}$. In the last row the rate is $N^{-2}$.
    \item The remaining rows show that the observed rates of $\sigmasqcvest$ to $0$ are in complete agreement with the rates predicted by \Cref{res:holder-spaces,res:holder-spaces-exp}.
 In particular, the rates are adaptive to the smoothness $l[f] + \alpha [f]$ of the function if $l[f] + \alpha[f] \leq 3/2$, as predicted.
\end{itemize}

\subsection{Comparison of CV and ML estimators}
\label{sec:comparison-CV-ML}

\Cref{fig:cv-vs-ml} shows the decay rates of $\sigmasqcvest$ and $\hat{\sigma}^2_\textup{ML}$ to $0$ for test functions $f$ with $l[f] = 1$, under the same setting as for \Cref{fig:cv}. In this case, \Cref{res:holder-spaces-ml,res:holder-spaces-exp-ml} predict that  $\sigmasqmlest$ decays at the rate $\Theta(N^{-1})$  regardless of the smoothness; this is confirmed in the right column. In contrast, the middle column shows again that $\sigmasqcvest$ decays with a rate that adapts to  $l[f]$ and $\alpha[f]$ as long as $l[f] + \alpha[f] \leq 3/2$, as predicted by \Cref{res:holder-spaces,res:holder-spaces-exp}. These results empirically support our theoretical finding that the CV estimator is adaptive to the unknown smoothness $l[f] + \alpha[f]$ of a function $f$ for a broader range of smoothness than the ML estimator.

Additionally, in~\Cref{sec:cv-vs-ml-sobolev}, we compare the asymptotics of the CV and ML estimators when the underlying kernel is a Mat\'ern kernel and the Sobolev smoothness of the true functions differs from that of the kernel. Similarly to the results presented in this section, we observe that the CV estimator exhibits a larger range of adaptation than the ML estimator.

\begin{figure}
    \centering
    \includegraphics[width=\textwidth]{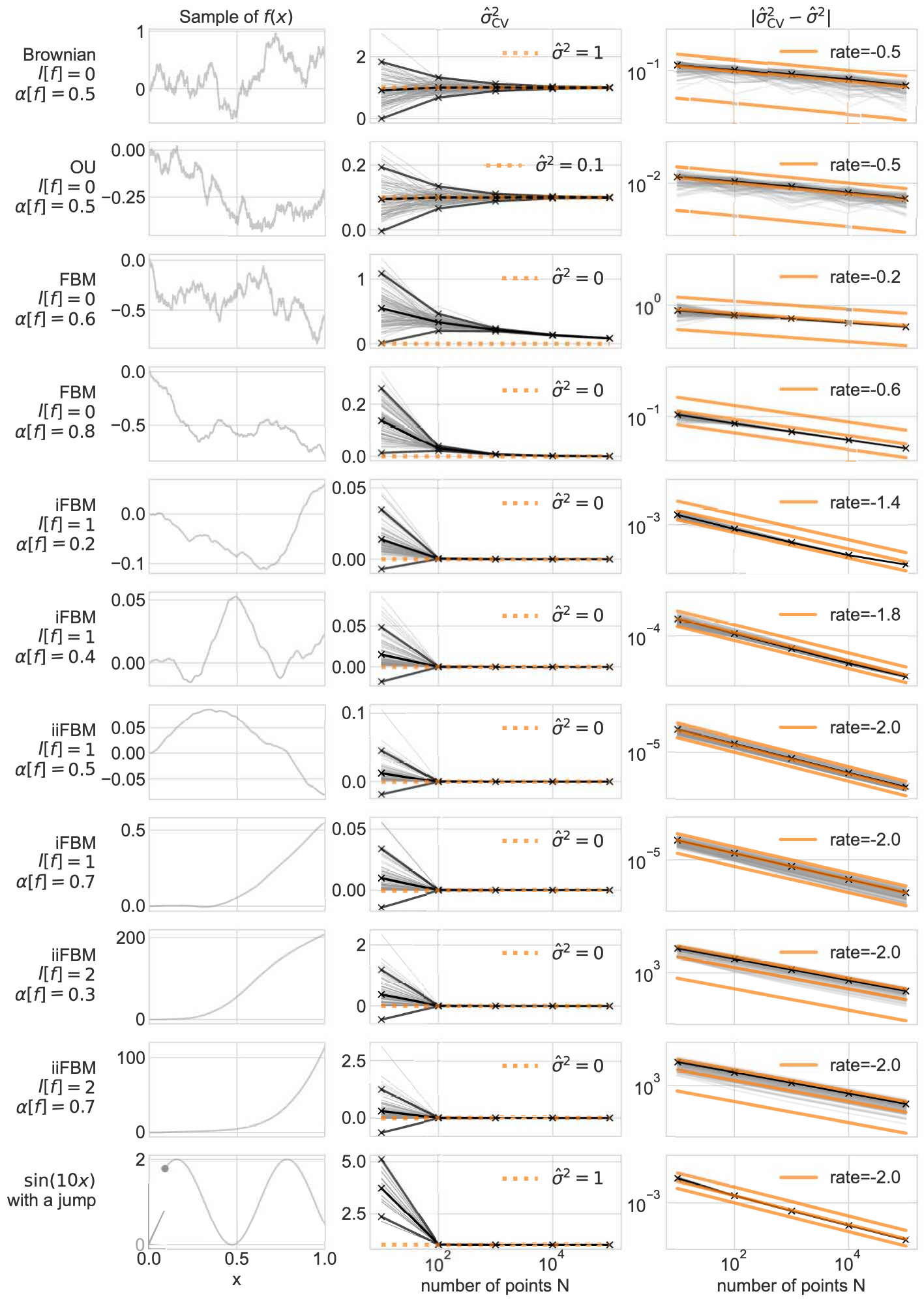}
    \caption{Asymptotics of CV estimators for functions of varying smoothness as quantified by $l[f]$ and $\alpha[l]$ in \eqref{eq:smoothness-1037}. Runs on individual 100 samples from $f$ are in gray, means and confidence intervals (of two standard deviations) are in black.}
    \label{fig:cv}
\end{figure}
\begin{figure}
    \centering
    \includegraphics[width=\textwidth]{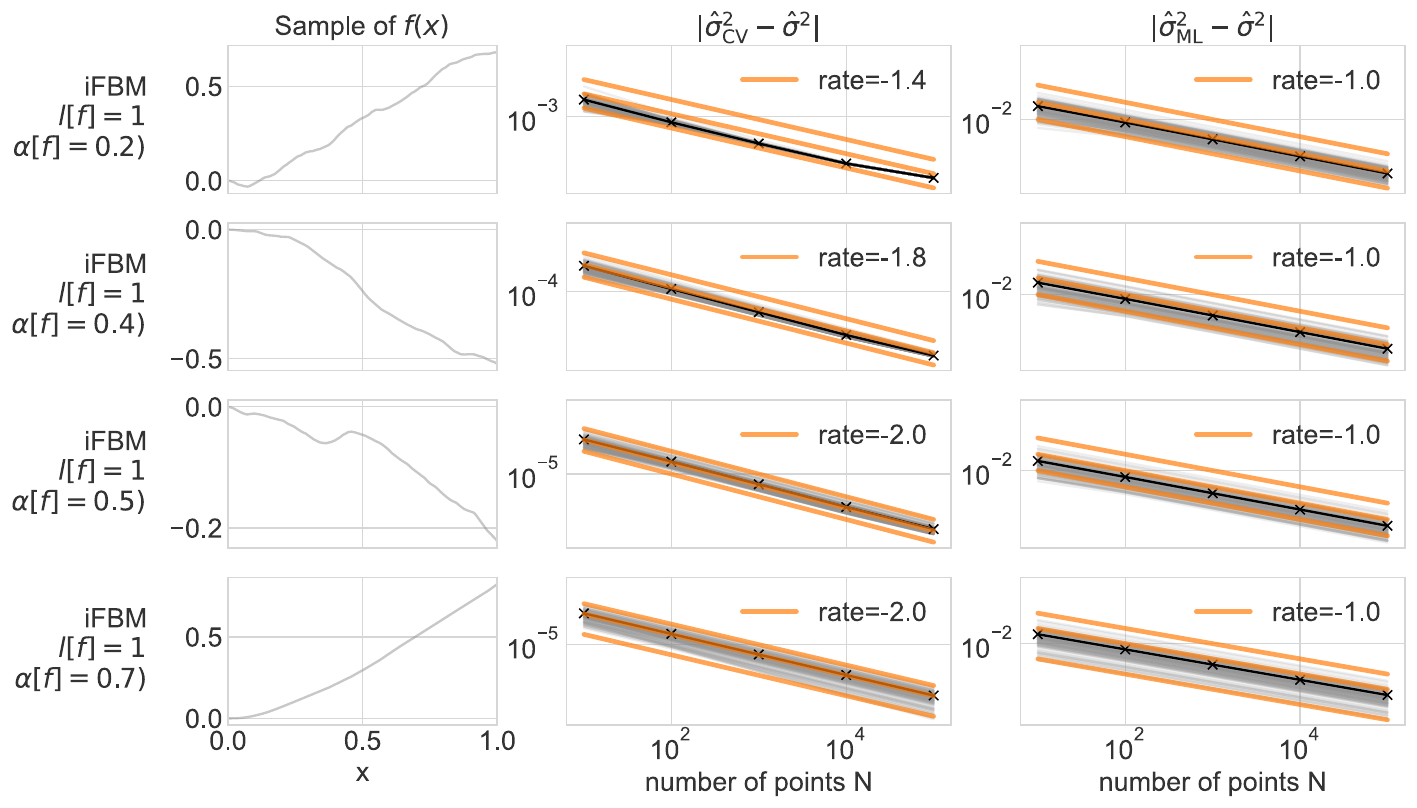}
    \caption{Asymptotics of CV estimator compared to asymptotics of ML estimators, for once differentiable functions.}
    \label{fig:cv-vs-ml}
\end{figure}

\section{Conclusion and future work}
\label{sec:conclusion}

We have analysed the asymptotics of the CV and ML estimators for the scale parameter in GP interpolation with the Brownian motion kernel. As a novel contribution, our analysis covers the misspecified case where the smoothness of the true function $f$ is different from that of the samples from the GP prior. Our main results in \cref{res:holder-spaces,res:holder-spaces-ml,res:holder-spaces-exp,res:holder-spaces-exp-ml} indicate that both CV and ML estimators can adapt to the unknown smoothness of $f$, but the range of smoothness for which this adaptation happens is broader for the CV estimator. Accordingly, the CV estimator can make GP uncertainty estimates asymptotically well-calibrated for a wider range of smoothness than the ML estimator, as indicated in \Cref{res:uq-theorem-exp}. In this sense, the CV estimator has an advantage over the ML estimator. The experiments provide supporting evidence for the theoretical results.

Natural next steps include the following:
\begin{itemize}
    \item Supplement the asymptotic upper bounds in \Cref{res:holder-spaces,res:holder-spaces-ml} of the deterministic setting with matching lower bounds.
    \item Extend the analyses (of both the deterministic and random settings) to more generic finitely smooth kernels, higher dimensions, and a noisy setting.
\end{itemize}

The matching lower bounds, if obtained, would enable the analysis of the ratio between the prediction error $|f(x) - m_N(x)|$ and the posterior standard deviation $\hat{\sigma} \sqrt{\smash[b]{k_N(x)}}$ in the deterministic setting, corresponding to the one in \Cref{sec:discussion} for the random setting.  Such an analysis would need additional assumptions on the true function $f$, such as the homogeneity of the smoothness of $f$ across the input space.  It also requires a sharp characterisation of the error $| f(x) - m_N(x) |$, which could use super convergence results in \citet[Section~11.5]{Wendland2005} and \citet{Schaback2018}.
Most natural kernel classes for extension are Mat\'erns and other kernels whose RKHS are norm-equivalent to Sobolev spaces; we conduct initial empirical analysis in~\Cref{sec:cv-vs-ml-sobolev} and observe results consistent with the main results in this paper. To this end, it would be possible to adapt the techniques used in \citet{Karvonen2020} for analysing the ML estimator to the CV estimator. In any case, much more advanced techniques
than those used here would be needed. A potentially more straightforward extension could be one to multiple times integrated Brownian motion kernels for which Gaussian process interpolation corresponds to spline interpolation~\citep[Chapter~1]{Wahba1990}.
In particular, finding analytic expression for the mean and variance of a cubic spline kernel given in, for example, Equation~(6.28) of \citet{rassmussen2006gaussian} can be reduced to the problem of inverting a tridiagonal matrix targeted in~\citet{mallik2001inverse} and~\citet{kilicc2008explicit}.

\section{Proofs} \label{sec:proofs}

This section provides the proofs of the main results and other lengthy computations.
For $x_0 = 0$ and $x_1, \dots, x_N \in [0,T]$, we will use the following notation whenever it can improve the readability or highlight a point:
\begin{align}
   \Delta x_n &\coloneqq x_{n+1} - x_n, \quad n  = 0, 1, \dots, N-1,  \nonumber \\
   f_n &\coloneqq f(x_n), \quad n = 0, 1, \dots, N. \label{eq:notation-x-f}
\end{align}

%All proofs and other lengthy computations are given in this section.

\subsection{Explicit expressions for the CV and ML estimators} \label{sec:explicit-post-mean-cov}
Let us define $x_0 = 0$ and use the convention $f(x_0) = 0$.
By a direct computation it is straightforward to verify that the inverse of the Gram matrix of the Brownian motion kernel $k(x, x') = \min(x, x')$ over the points $0=x_0<x_1 < x_2 < \dots < x_N$ is the band matrix
\begin{align*}
  k(\bx, \bx)^{-1} =
  \begin{bmatrix}
      x_1 & x_1 & x_1 & \dots & x_1 & x_1 \\
      x_1 & x_2 & x_2 & \dots & x_2 & x_2 \\
      x_1 & x_2 & x_3 & \dots & x_3 & x_3 \\
      \vdots & \vdots & \vdots  & \ddots  & \vdots & \vdots \\
      x_1 & x_2 & x_3 & \dots & x_{N-1} & x_{N-1} \\
      x_1 & x_2 & x_3 & \dots & x_{N-1} & x_N
    \end{bmatrix}^{-1}
    =
    \begin{bmatrix}
        b_1    & c_1    & 0          & \dots  & 0       & 0      \\
        c_1    & b_2    & c_2        & \dots  & 0       & 0      \\
        0      & c_2    & b_3        & \dots  & 0       & 0      \\
        \vdots & \vdots & \vdots & \ddots & \vdots  & \vdots \\
        0      & 0          & 0      & \dots  & b_{N-1} & c_{N-1}\\
        0      & 0          & 0      & \dots  & c_{N-1} & b_N    \\
    \end{bmatrix} ,
\end{align*}
where
\begin{align*}
    b_i &= \frac{x_{i+1} - x_{i-1}}{(x_{i-1} - x_i) (x_i - x_{i+1})} \quad \text{for} \quad i \in \{2,\dots,N-1\}, \qquad b_N &= - \frac{1}{x_{N-1} - x_{N}}, \\
    c_i &= \frac{1}{(x_i - x_{i+1})} \quad \text{for} \quad i \in \{1,\dots,N-1\}. \\
\end{align*}
It follows that the posterior mean and covariance functions in \eqref{eq:posterior-moments} can be expressed as
\begin{equation} \label{eq:explicit-post-mean}
  m_N(x) =
  \begin{dcases}
    \frac{ (x_n - x) f(x_{n-1}) + (x - x_{n-1}) f(x_{n}) }{x_n - x_{n-1}} &\hspace{-0.2cm}\text{ if } x \in [x_{n-1}, x_n] \text{ for some } 1 \leq n \leq N, \\
    f(x_N) &\hspace{-0.2cm}\text{ if } x \in [x_N, T]
  \end{dcases}
\end{equation}
and
\begin{equation} \label{eq:explicit-post-cov}
  k_N(x, x') =
  \begin{dcases}
    \frac{ ( x_n - x') ( x - x_{n-1})  }{x_n - x_{n-1}} &\text{ if } x_{n-1} \leq x \leq x' \leq x_n \text{ for some } 1 \leq n \leq N, \\
    x - x_N &\text{ if } x_N \leq x \leq x' \leq T, \\
    0 &\text{ otherwise}.
  \end{dcases}
\end{equation}
We omit the case $x' \leq x$ for $k_N(x,x')$ as this case is  obtained by the symmetry $k_N(x,x') = k_N(x', x)$.

Using these expressions, we have, for each $1 \leq n < N$:
%By applying the above equations we obtain
%
\begin{equation*}
  m_{\setminus n}(x_n) = \frac{ (x_n - x_{n + 1}) f(x_{n-1}) + (x_{n - 1} - x_{n}) f(x_{n+1}) }{x_{n-1} - x_{n + 1}}
\end{equation*}
and
\begin{equation*}
  k_{\setminus n} (x_n) = k_{\setminus n}(x_n, x_n) = \frac{ ( x_n - x_{n + 1}) ( x_n - x_{n - 1}) }{x_{n-1} - x_{n + 1}}
\end{equation*}
%for each $1 \leq n < N$, a
For $n = N$, we have $m_{\setminus N}(x_N) = f(x_{N-1})$ and $k_{\setminus N}(x_N) = x_N - x_{N-1}$.
Inserting these expressions in \eqref{eq:sigma-cv} and using the notation \eqref{eq:notation-x-f}, the CV estimator can be written as
\begin{equation}
\label{eq:sigma-cv-w-deltas}
\begin{split}
    \sigmasqcvest &= \frac{1}{N} \Bigg[ \frac{(x_2 f_1 -  x_1 f_2  )^2}{ x_1 x_2 \Delta x_1 } + \sum_{n=2}^{N-1}  \frac{( \Delta x_{n-1} [f_{n+1} - f_n] - \Delta x_n [f_n - f_{n-1}] )^2}{ (\Delta x_n + \Delta x_{n-1}) \Delta x_n \Delta x_{n-1} } \\
    &\hspace{1.2cm}+ \frac{(f_N - f_{N-1})^2}{ \Delta x_{N-1} } \Bigg].
\end{split}
\end{equation}

For the ML estimator \eqref{eq:sigma-ml}, we obtain the explicit expression
\begin{equation}
  \label{eq:sigma-ml-w-deltas}
  \sigmasqmlest = \frac{1}{N} \sum_{n=1}^N \frac{[ f(x_n) - f(x_{n-1}) ]^2}{ \Delta x_{n-1} }
\end{equation}
by observing that $m_{n-1}(x_n) = f(x_n)$ and $k_{n-1}(x_n) = x_n - x_{n-1}$.

\begin{remark}
    The leave-$p$-out estimator $\hat{\sigma}_{\textup{CV}(p)}^2$ can be expressed in a form similar (albeit more complicated) to~\Cref{eq:sigma-cv-w-deltas}. We derive this expression in~\Cref{sec:explicit_expression_for_leave_p_out}. This suggests that the analysis in~\Cref{sec:limit-behaviour-for-sigma} could potentially be generalised to apply to the leave-$p$-out estimators, a possibility that we leave open for future research to explore.
\end{remark}

\subsection{Proofs for Section~\ref{sec:results-deterministic}} \label{sec:proofs-deterministic}

\begin{proof}[Proof of Theorem~\ref{res:holder-spaces}]
The estimator $\sigmasqcvest$ in~\eqref{eq:sigma-cv-w-deltas} may be written as
\begin{equation} \label{eq:sigma-cv-boundary-interior}
  \sigmasqcvest = B_{1, N} + I_N + B_{2,N}
\end{equation}
in terms of the boundary terms
\begin{equation} \label{eq:boundary-terms}
  B_{1,N} = \frac{1}{N} \cdot \frac{(x_2 f_1 -  x_1 f_2  )^2}{ x_1 x_2 \Delta x_1 } \quad \text{ and } \quad B_{2,N} = \frac{1}{N} \cdot \frac{(f_N - f_{N-1})^2}{ \Delta x_{N-1} }
\end{equation}
and the interior term
\begin{equation} \label{eq:interior-term}
    I_N = \frac{1}{N} \sum_{n=2}^{N-1}  \frac{( \Delta x_{n-1} [f_{n+1} - f_n] - \Delta x_n [f_n - f_{n-1}] )^2}{ (\Delta x_n + \Delta x_{n-1}) \Delta x_n \Delta x_{n-1} } .
\end{equation}
The claimed rate in~\eqref{eq:main-result} is $\bigo( N^{-2} )$ if $l \geq 2$ or $l = 1$ and $\alpha \geq 1/2$.
By the inclusion properties of H\"older spaces in Section \ref{sec:holder-space}, it is therefore sufficient to consider the cases (a) $l = 0$ and (b) $l=1$ and $\alpha \in (0, 1/2]$.

Suppose first that $l = 0$.
Let $L$ be a H\"older constant of a function $f \in C^{0, \alpha}([0,T])$.
Using the H\"older condition, the bounding assumption on $\Delta x_n$, and $f_0 = f(0) = 0$, the boundary terms can be bounded as
\begin{align}
  B_{1,N} = \frac{1}{N} \cdot \frac{ (x_1 (f_1 - f_2) + \Delta x_1 (f_1 - f_0) )^2}{x_1 x_2 \Delta x_1} &\leq \frac{1}{N} \cdot \frac{2(x_1^2 (f_1 - f_2)^2 + \Delta x_1^2 (f_1 - f_0)^2)}{ x_1 x_2 \Delta x_1 } \nonumber \\
    &\leq \frac{1}{N} \cdot \frac{ 2L^2 (x_1^2 \Delta x_1^{2\alpha} + x_1^{2\alpha} \Delta x_1^2  ) }{x_1 x_2 \Delta x_1 } \nonumber \\
    & = \bigo(N^{-1} \Delta x_1^{2\alpha - 1}) \nonumber \\
    & = \bigo(N^{-2 \alpha}) \label{eq:bound-B1N}
\end{align}
and
\begin{equation}
  B_{2,N} = \frac{1}{N} \cdot \frac{(f_N - f_{N-1})^2}{ \Delta x_{N-1} } \leq \frac{1}{N} L^2 \Delta x_{N-1}^{2\alpha - 1} = \bigo(N^{- 2\alpha}). \label{eq:bound-B2N}
\end{equation}
Similarly, the interior term is bounded as
\begin{align*}
  I_N &\leq
  \frac{2}{N}\sum_{n=2}^{N-1}  \frac{\Delta x_{n-1}^2 (f_{n+1} - f_n)^2 + \Delta x_n^2 ( f_n - f_{n-1})^2}{ (\Delta x_n + \Delta x_{n-1}) \Delta x_n \Delta x_{n-1} } \\
  &\leq \frac{2L^2}{N} \sum_{n=2}^{N-1}  \frac{\Delta x_{n-1}^2 \Delta x_n^{2\alpha} + \Delta x_n^2 \Delta x_{n-1}^{2\alpha} }{ (\Delta x_n + \Delta x_{n-1}) \Delta x_n \Delta x_{n-1} }
  \\
    & = \frac{2L^2}{N} \sum_{n=2}^{N-1}  \frac{ \Delta x_{n-1} \Delta x_n^{2\alpha - 1} + \Delta x_n \Delta x_{n-1}^{2\alpha - 1}  }{ \Delta x_n + \Delta x_{n-1}} \\
    &= \frac{2L^2}{N} \sum_{n=2}^{N-1}  \bigg( \frac{ \Delta x_{n-1}}{ \Delta x_n + \Delta x_{n-1}} \Delta x_n^{2\alpha - 1} + \frac{\Delta x_n }{ \Delta x_n + \Delta x_{n-1}} \Delta x_{n-1}^{2\alpha - 1} \bigg) \\
    & \leq \frac{2L^2}{N} \sum_{n=2}^{N-1} \big( \Delta x_n^{2\alpha - 1} + \Delta x_{n-1}^{2\alpha - 1} \big) \\
    &= \bigo (N^{1 - 2\alpha}).
\end{align*}
Inserting the above bounds in~\eqref{eq:sigma-cv-boundary-interior} yields $\sigmasqcvest = \bigo(N^{-2\alpha} + N^{1 - 2\alpha}) = \bigo(N^{1-2\alpha})$, which is the claimed rate when $l=0$.

Suppose then that $l = 1$ and $\alpha \in (0, 1/2]$, so that the first derivative $f'$ of $f \in C^{1, \alpha}([0, T])$ is $\alpha$-H\"older and hence continuous.
Because a continuously differentiable function is Lipschitz, we may set $\alpha = 1$ in the estimates \eqref{eq:bound-B1N} and \eqref{eq:bound-B2N} for the boundary terms $B_{1,N}$ and $B_{2,N}$ in the preceding case.
This shows these terms are $\bigo(N^{-2})$.
Because $f$ is differentiable, we may use the mean value theorem to write the interior term as
\begin{align*}
    I_N &= \frac{1}{N} \sum_{n=2}^{N-1}  \frac{\Delta x_{n-1} \Delta x_n}{\Delta x_{n-1} + \Delta x_n} \bigg( \frac{f_{n+1} - f_n}{\Delta x_n} - \frac{f_n - f_{n-1}}{\Delta x_{n-1}} \bigg)^2 \\
    &= \frac{1}{N} \sum_{n=2}^{N-1}  \frac{\Delta x_{n-1} \Delta x_n}{\Delta x_{n-1} + \Delta x_n} \big[ f'(\tilde{x}_n) - f'(\tilde{x}_{n-1}) \big]^2,
\end{align*}
where $\tilde{x}_n \in (x_{n}, x_{n+1})$.
Let $L$ be a H\"older constant of $f'$.
Then the H\"older continuity of $f'$ and the assumption that $\|\prt_N \|=\bigo(N^{-1})$ yield
\begin{align*}
    I_N \leq \frac{L^2}{N} \sum_{n=2}^{N-1}  \frac{\Delta x_{n-1} \Delta x_n}{\Delta x_{n-1} + \Delta x_n} \lvert \tilde{x}_n - \tilde{x}_{n-1} \rvert^{2\alpha} &\leq \frac{L^2}{N} \sum_{n=2}^{N-1}  \frac{\Delta x_{n-1} \Delta x_n}{\Delta x_{n-1} + \Delta x_n} (\Delta x_{n-1} + \Delta x_n)^{2\alpha} \\
    &\leq \frac{L^2}{N} \sum_{n=2}^{N-1} \Delta x_n (\Delta x_{n-1} + \Delta x_n)^{2\alpha} \\
    &= \bigo(N^{-2\alpha - 1}).
\end{align*}
Using the above bounds in~\eqref{eq:sigma-cv-boundary-interior} yields $\sigmasqcvest = \bigo(N^{-2} + N^{-2\alpha-1}) = \bigo(N^{-2\alpha-1})$, which is the claimed rate when $l=1$.
\end{proof}

\begin{proof}[Proof of Theorem~\ref{res:holder-spaces-ml}]

  From~\eqref{eq:sigma-ml-w-deltas} we have
  \begin{equation*}
    \sigmasqmlest = \frac{1}{N} \sum_{n=1}^N \frac{(f_n - f_{n-1})^2}{\Delta x_{n-1}}.
  \end{equation*}
  Suppose first that $l = 0$.
  As in the proof of \Cref{res:holder-spaces}, we get
  \begin{equation} \label{eq:holder-spaces-ml-derivation}
    \sigmasqmlest = \frac{1}{N} \sum_{n=1}^N \frac{(f_n - f_{n-1})^2}{\Delta x_{n-1}} \leq \frac{L^2}{N} \sum_{n=1}^N \Delta x_{n-1}^{2\alpha - 1} = \bigo\big( N^{1-2\alpha} \big)
  \end{equation}
  when $\|\prt_N \|=\bigo(N^{-1})$.
  Suppose then that $l = 1$.
  By the mean value theorem there are $\xi_n \in (x_{n-1}, x_n)$ such that
  \begin{equation*}
    \sigmasqmlest = \frac{1}{N} \sum_{n=1}^N \frac{(f_n - f_{n-1})^2}{\Delta x_{n-1}} = \frac{1}{N} \sum_{n=1}^N \Delta x_{n-1} \bigg( \frac{f_n - f_{n-1}}{\Delta x_{n-1}} \bigg)^2 = \frac{1}{N} \sum_{n=1}^N \Delta x_{n-1} f'(\xi_n)^2.
  \end{equation*}
  Since $f'$ is continuous on $[0, T]$ and hence Riemann integrable, we obtain the asymptotic equivalence
  \begin{equation*}
    N \sigmasqmlest \to \int_0^T f'(x)^2 \, \mathrm{d} x \quad \text{ as } \quad N \to \infty
  \end{equation*}
  when $\|\prt_N \| \to 0$ as $N \to \infty$.
  The integral is positive because $f$ has been assumed non-constant.
\end{proof}

\begin{proof}[Proof of Theorem~\ref{res:fqv-estimator}]
For equally-spaced partitions, $\Delta x_n = x_1 = T/N$ for all $n \in \{0, \dots, N-1\}$, the estimator $\sigmasqcvest$ in~\eqref{eq:sigma-cv-w-deltas} takes the form
\begin{equation*}
    \sigmasqcvest = \frac{1}{T} \Bigg[ \frac{(x_2 f_1 -  x_1 f_2  )^2}{ x_1 x_2} + \frac{1}{2}\sum_{n=2}^{N-1} (  (f_{n+1} - f_n) - (f_n - f_{n-1}) )^2 + (f_N - f_{N-1})^2\Bigg].
\end{equation*}
Recall from the proof of \Cref{res:holder-spaces} the decomposition
\begin{equation*}
  \sigmasqcvest = B_{1, N} + I_N + B_{2,N}
\end{equation*}
in terms of the boundary terms $B_{1,N}$ and $B_{2,N}$ in~\eqref{eq:boundary-terms} and the interior term $I_N$ in~\eqref{eq:interior-term}.
Because $f$ is assumed continuous on the boundary and equispaced partitions are quasi-uniform, both $B_{1,N}$ and $B_{2,N}$ tend to zero as $N \to \infty$.
We may therefore focus on the interior term, which decomposes as
\begin{equation*}
\begin{split}
    I_N = {}& \frac{1}{2}\sum_{n=2}^{N-1} \left(  (f_{n+1} - f_n) - (f_n - f_{n-1}) \right)^2\\
    ={}& \sum_{n=2}^{N-1} (f_{n+1} - f_n)^2
    + ( f_n - f_{n-1} )^2 - \frac{1}{2} ( f_{n+1} - f_{n-1} )^2
\end{split}
\end{equation*}
The sums $\sum_{n=2}^{N-1}  (f_{n+1} - f_n)^2$ and $\sum_{n=2}^{N-1}  ( f_n - f_{n-1} )^2$ tend to $V^2(f)$ by definition. To establish the claimed bound we are therefore left to prove that
\begin{equation}
\label{eq:requirement-double-prt-to-v2f}
    \sum_{n=2}^{N-1} (f_{n+1} - f_{n-1} )^2 \to 2V^2(f) \qquad \text{as} \qquad N \to \infty.
\end{equation}
We may write the sum as
\begin{align*}
  \sum_{n=2}^{N-1} (f_{n+1} - f_{n-1} )^2 &= \sum_{n=1}^{\lfloor \frac{N-1}{2} \rfloor} (f_{2n+1} - f_{2n-1} )^2 + \sum_{n=1}^{\lfloor \frac{N-2}{2} \rfloor} (f_{2n+2} - f_{2n} )^2.
\end{align*}
Consider a sub-partition of $\prt_N$ that consists of odd-index points $x_1, x_3, \dots x_{2\lfloor \frac{N-1}{2} \rfloor+1} $ of $\prt_N$. The sequence of these sub-partitions is quasi-uniform with constant $2$. The assumption that the quadratic variation is $V^2(f)$ for all partitions with quasi-uniformity constant $2$ implies that
\begin{align*}
  \lim_{N \to \infty} \sum_{n=1}^{\lfloor \frac{N-1}{2} \rfloor} (f_{2n+1} - f_{2n-1} )^2 = V^2(f).
\end{align*}
The same will hold for sub-partitions formed of even-index points of $\prt_N$, giving
\begin{align*}
  \lim_{N \to \infty} \sum_{n=1}^{\lfloor \frac{N-2}{2} \rfloor} (f_{2n+2} - f_{2n} )^2 = V^2(f).
\end{align*}
Thus,~\eqref{eq:requirement-double-prt-to-v2f} holds. This completes the proof.
\end{proof}

\subsection{Proofs for Section~\ref{sec:random-setting}} \label{sec:proofs-random}

\begin{proof}[Proof of Theorem~\ref{res:holder-spaces-exp}]
    Recall the explicit expression of $\sigmasqcvest$ in~\eqref{eq:sigma-cv-w-deltas}:
    \begin{equation}
    \begin{split} \label{eq:expnasion-1056}
        \sigmasqcvest &= \frac{1}{N} \Bigg[ \frac{(x_2 f_1 -  x_1 f_2  )^2}{ x_1 x_2 \Delta x_1 } + \sum_{n=2}^{N-1}  \frac{( \Delta x_{n-1} [f_{n+1} - f_n] - \Delta x_n [f_n - f_{n-1}] )^2}{ (\Delta x_n + \Delta x_{n-1}) \Delta x_n \Delta x_{n-1} } \\
        &\hspace{1.2cm}+ \frac{(f_N - f_{N-1})^2}{ \Delta x_{N-1} } \Bigg].
    \end{split}
    \end{equation}
    We consider the cases $l = 0$ and $l = 1$ separately. Recall that $f \sim \GP(0, k_{l, H})$ implies that $\E[f(x)f(x')] = k_{l,H}(x, x')$.

    Suppose first that $l = 0$, in which case $f \sim \GP(0, k_{0, H})$ for the fractional Brownian motion kernel $k_{0,H}$ in~\eqref{eq:fbm-def}.
    In this case the expected values of squared terms in the expression for $\sigmasqcvest$ are $\E [x_2 f_1 -  x_1 f_2  ]^2 = x_1 x_2 \Delta x_1  (x_1^{2H - 1} + \Delta x_1^{2H-1}- (x_1+\Delta x_1)^{2H - 1} )$,
    \begin{align*}
      \E \big[ &\Delta x_{n-1} (f_{n+1} - f_n) - \Delta x_n (f_n - f_{n-1}) \big]^2  \\
      &= \big( \Delta x_n^{2H-1} +  \Delta x_{n-1}^{2H-1} - (\Delta x_{n-1} +\Delta x_n)^{2H-1} \big) \Delta x_{n-1}  \Delta x_n (\Delta x_n + \Delta x_{n-1} ),
    \end{align*}
    and $\E [f_N - f_{N-1}]^2 = \Delta x_{N-1}^{2H}$.
    Substituting these in the expectation of $\sigmasqcvest$ and using the fact that $\Delta x_n=\Theta(N^{-1})$ for all $n$ by quasi-uniformity we get
    \begin{equation*}
    \begin{split}
        \E \sigmasqcvest &= \frac{1}{N} \Bigg[ (x_1^{2H - 1} + \Delta x_1^{2H-1}- (x_1+\Delta x_1)^{2H - 1} ) \\
        &\hspace{1.2cm} + \sum_{n=2}^{N-1}  \left(\Delta x_{n-1}^{2H-1}  + \Delta x_n^{2H-1} - (\Delta x_{n-1} + \Delta x_n)^{2H-1} \right) + \Delta x_{N-1}^{2H - 1}\Bigg] \\
        &= \frac{1}{N} \Bigg[ \Delta x_1^{2H-1}\left(\left(\frac{x_1}{\Delta x_1}\right)^{2H - 1} + 1 - \left(\frac{x_1}{\Delta x_1^{2H-1}} + 1\right)^{2H - 1} \right) \\
        &\hspace{1.2cm} + \Delta x_n^{2H-1} \sum_{n=2}^{N-1}  \left(\left(\frac{\Delta x_{n-1}}{\Delta x_n}\right)^{2H-1}  + 1 - \left(\frac{\Delta x_{n-1}}{\Delta x_n} + 1\right)^{2H-1} \right) + \Delta x_{N-1}^{2H - 1}\Bigg] \\
        &\eqqcolon{}\frac{1}{N} \Bigg[ \Delta x_1^{2H-1} c_1 + \Delta x_n^{2H-1} \sum_{n=2}^{N-1} c_n + \Delta x_{N-1}^{2H - 1}\Bigg] .
    \end{split}
    \end{equation*}
    Notice that the function $x \mapsto x^{2H-1} + 1 - (x+1)^{2H-1}$ is positive for $x>0$ and $H \in (0, 1)$, and increasing for $H \in (1/2, 1)$ and non-increasing for $H \in (0, 1/2]$. By quasi-uniformity we have $C_\textup{qu}^{-1} \leq \Delta x_{n-1}/\Delta x_n \leq C_\textup{qu}$, and can bound $c_n$ for any $n$ and $N$ as
    \begin{equation*}
    \begin{split}
        0< C_\textup{qu}^{2H-1} + 1 - (C_\textup{qu}+1)^{2H-1} \leq c_n \leq C_\textup{qu}^{1-2H} + 1 - (C_\textup{qu}^{-1}+1)^{2H-1} \: &\text{ if } \:H \in (0, 1/2], \\
        0< C_\textup{qu}^{1-2H} + 1 - (C_\textup{qu}^{-1}+1)^{2H-1} \leq c_n \leq C_\textup{qu}^{2H-1} + 1 - (C_\textup{qu}+1)^{2H-1} \: &\text{ if } \: H \in (1/2, 1).
    \end{split}
    \end{equation*}
    Finally, by quasi-uniformity $\Delta x_n=\Theta(N^{-1})$, and $\E \sigmasqcvest=\Theta(N^{-2H}) + \Theta(N^{1-2H})+\Theta(N^{-2H})=\Theta(N^{1-2H})$.

    Suppose then that $l = 1$, in which case $f \sim \GP(0, k_{1, H})$ for the integrated fractional Brownian motion kernel $k_{1,H}$ in~\eqref{eq:iFBM-kernel-explicit}.
    It is straightforward (though, in the case of the second expectation, somewhat tedious) to compute that the expected values of squared terms in the expression \eqref{eq:expnasion-1056} for $\sigmasqcvest$ are
    \begin{equation*}
        \E [x_2 f_1 -  x_1 f_2  ]^2 = \frac{x_1 x_2 \Delta x_1 }{2(H+1)(2H+1)} \big( x_2^{2H+1} - x_1^{2H+1} - \Delta x_1^{2H+1} \big)
    \end{equation*}
    and
    \begin{equation} \label{eq:l=1-exp-2}
      \begin{split}
        \E \big[ \Delta &x_{n-1} (f_{n+1} - f_n) - \Delta x_n (f_n - f_{n-1}) \big]^2  \\
        ={}& \frac{\Delta x_n \Delta x_{n-1} (\Delta x_{n} + \Delta x_{n-1}) }{2(H+1)(2H+1)} \big[ (\Delta x_{n} + \Delta x_{n-1})^{2H+1} - \Delta x_n^{2H+1} - \Delta x_{n-1}^{2H+1} \big]
      \end{split}
    \end{equation}
    and
    \begin{equation*}
      \begin{split}
        \E [f_N - f_{N-1}]^2 = \frac{\Delta x_{N-1}}{2H+1} \bigg( x_N^{2H+1} - x_{N-1}^{2H+1} - \frac{1}{2(H+1)} \Delta x_{N-1}^{2H+1} \bigg).
      \end{split}
    \end{equation*}
    Therefore, by \eqref{eq:expnasion-1056},
    \begin{equation*}
      \begin{split}
        \E \sigmasqcvest ={}& \frac{\big( x_2^{2H+1} - x_1^{2H+1} - \Delta x_1^{2H+1} \big)}{2(H+1)(2H+1)N} \\
        &+ \frac{1}{2(H+1)(2H+1)N}\sum_{n=2}^{N-1} \big[ (\Delta x_{n} + \Delta x_{n-1})^{2H+1} - \Delta x_n^{2H+1} - \Delta x_{n-1}^{2H+1} \big] \\
        &+ \frac{1}{(2H+1)N} \bigg( x_N^{2H+1} - x_{N-1}^{2H+1} - \frac{1}{2(H+1)} \Delta x_{N-1}^{2H+1} \bigg) \\
        \eqqcolon{}& \frac{1}{2(H+1)(2H+1)} B_{1,N} + \frac{1}{2(H+1)(2H+1)} I_N + \frac{1}{(2H+1)} B_{2,N}.
      \end{split}
    \end{equation*}
    By quasi-uniformity, $B_{1,N} \leq N^{-1} x_2^{2H+1} = \bigo(N^{-2-2H})$.
    Consider then the interior term
    \begin{equation} \label{eq:interior-term-exp-cv-proof}
      \begin{split}
      I_N &= \frac{1}{N} \sum_{n=2}^{N-1} \Delta x_n^{2H+1} \Bigg[ \bigg(1 + \frac{\Delta x_{n-1}}{\Delta x_n} \bigg)^{2H+1} - \bigg(1 + \bigg(\frac{\Delta x_{n-1}}{\Delta x_n} \bigg)^{2H+1} \bigg) \Bigg] \\
      &\eqqcolon \frac{1}{N} \sum_{n=2}^{N-1} \Delta x_n^{2H+1} c'_n.
      \end{split}
    \end{equation}
    Because the function $x \mapsto (1 + x)^{2H+1} - (1 + x^{2H+1})$ is positive and increasing for $x > 0$ if $H \in (0, 1)$ and $C_\textup{qu}^{-1} \leq \Delta x_{n-1} / \Delta x_n \leq  C_\textup{qu}$ by quasi-uniformity, we have
    \begin{equation*}
      0 < (1 + C_\textup{qu}^{-1})^{2H+1} - (1 + C_\textup{qu}^{-(2H+1)}) \leq c'_n \leq \bigg(1 + \frac{\Delta x_{n-1}}{\Delta x_n} \bigg)^{2H+1} \leq (1 + C_\textup{qu})^{2H+1}
    \end{equation*}
    for every $n$.
    Because $N^{-1} \sum_{n=2}^{N-1} \Delta x_n^{2H+1} = \Theta(N^{-1-2H})$ by quasi-uniformity, we conclude from~\eqref{eq:interior-term-exp-cv-proof} that $I_N = \Theta(N^{-1-2H})$.
    For the last term $B_{2, N}$, recall that we have set $x_N = T$.
    Thus
    \begin{equation*}
      \begin{split}
        B_{2,N} &= \frac{1}{N} \bigg( T^{2H+1} - (T - \Delta x_{N-1})^{2H+1} - \frac{1}{2(H+1)} \Delta x_{N-1}^{2H+1} \bigg).
      \end{split}
    \end{equation*}
    By the generalised binomial theorem,
    \begin{equation*}
      T^{2H+1} - (T - \Delta x_{N-1})^{2H+1} = (2H+1) T^{2H} \Delta x_{N-1} + \bigo( \Delta x_{N-1}^2 )
    \end{equation*}
    as $\Delta x_{N-1} \to 0$.
    It follows that under quasi-uniformity we have $B_{2,N} = \Theta(N^{-2})$ for every $H \in (0, 1)$.
    Putting these bounds for $B_{1,N}$, $I_N$ and $B_{2,N}$ together we conclude that
    \begin{equation*}
        \begin{split}
        \E \sigmasqcvest &= \frac{1}{2(H+1)(2H+1)} B_{1,N} + \frac{1}{2(H+1)(2H+1)} I_N + \frac{1}{(2H+1)} B_{2,N} \\
        &= \bigo(N^{-2-2H}) + \Theta(N^{-1-2H}) + \Theta(N^{-2}),
        \end{split}
    \end{equation*}
    which gives $\E \sigmasqcvest = \Theta(N^{-1-2H})$ if $H \in (0, 1/2]$ and $\E \sigmasqcvest = \Theta(N^{-2})$ if $H \in [1/2, 1)$.
\end{proof}

Observe that in the proof of Theorem~\ref{res:holder-spaces-exp} it is the boundary term $B_{2,N}$ that determines the rate when there is sufficient smoothness, in that $l = 1$ and $H \in [1/2, 1)$.
Similar phenomenon occurs in the proof of Theorem~\ref{res:holder-spaces}.
The smoother a process is, the more correlation there is between its values at far-away points.
Because the Brownian motion (as well as fractional and integrated Brownian motions) has a zero boundary condition at $x = 0$ but no boundary condition at $x = T$ and no information is available at points beyond $T$, the importance of $B_{2,N}$ is caused by the fact that around $T$, the least information about the process is available.

\begin{proof}[Proof of Theorem~\ref{res:holder-spaces-exp-ml}]
  From~\eqref{eq:sigma-ml-w-deltas} we get
  \begin{equation*}
    \E \hat{\sigma}_\textup{ML}^2 = \frac{1}{N} \sum_{n=1}^N \frac{\E[f_n - f_{n-1}]^2}{\Delta x_{n-1}}.
  \end{equation*}
  We may then proceed as in the proof of Theorem~\ref{res:holder-spaces-exp} and use quasi-uniformity to show that
  \begin{equation*}
    \E \hat{\sigma}_\textup{ML}^2 = \frac{1}{N} \sum_{n=1}^N \frac{\E[f_n - f_{n-1}]^2}{\Delta x_{n-1}} = \frac{1}{N} \sum_{n=1}^N \frac{ \Delta x_{n-1}^{2H}}{\Delta x_{n-1}} = \frac{1}{N} \sum_{n=1}^N \Delta x_{n-1}^{2H-1} = \Theta(N^{1-2H})
  \end{equation*}
  when $l = 0$ and
  \begin{equation*}
    \begin{split}
      \E \hat{\sigma}_\textup{ML}^2 &= \sum_{n=1}^N \frac{\E[f_n - f_{n-1}]^2}{\Delta x_{n-1}} \\
      &= \frac{1}{(2H+1)N} \sum_{n=1}^N \bigg( x_n^{2H+1} - x_{n-1}^{2H+1} - \frac{1}{2(H+1)} \Delta x_{n-1}^{2H+1} \bigg) \\
      &= \frac{1}{(2H+1)N} \sum_{n=1}^N \bigg( (2H+1) x_n^{2H} \Delta x_{n-1} + \bigo(\Delta x_{n-1}^2) - \frac{1}{2(H+1)} \Delta x_{n-1}^{2H+1} \bigg) \\
      &= \Theta(N^{-1})
      \end{split}
  \end{equation*}
  when $l = 1$.
\end{proof}
\subsection{Proofs for~\Cref{sec:icv-estimators}} \label{sec:proofs-icv}

For the Brownian motion kernel, the ICV estimator defined in~\eqref{eq:icv-estimator} takes the explicit form
\begin{equation*}
    \sigmasqicvest = \frac{1}{N} \sum_{n=2}^{N-1}  \frac{( \Delta x_{n-1} [f_{n+1} - f_n] - \Delta x_n [f_n - f_{n-1}] )^2}{ (\Delta x_n + \Delta x_{n-1}) \Delta x_n \Delta x_{n-1} }.
\end{equation*}
We analyse this estimator below.

\begin{proof}[Proof of Theorem~\ref{res:holder-spaces-icv}]

The proof of~\Cref{res:holder-spaces} shows that when $l=1$ and $\alpha \in (1/2, 1]$, the bound is dominated by the bound on the boundary terms, $B_{1,N}=\bigo(N^{-2})$ and $B_{2,N}=\bigo(N^{-2})$, since
\begin{equation*}
    \sigmasqcvest = B_{1,N} + I_N + B_{2,N} = \bigo(N^{-2}) + \bigo(N^{-1-2\alpha}) + \bigo(N^{-2}) = \bigo(N^{-2}).
\end{equation*}
As $\sigmasqicvest = I_N$, it follows that $\sigmasqicvest = \bigo(N^{-1-2\alpha})$ when $l=1$.
\end{proof}

\begin{proof}[Proof of Theorem~\ref{res:holder-spaces-icv-exp}]

The proof of~\Cref{res:holder-spaces-exp} shows that when $l=1$ and $H \in [1/2, 1)$, the bound is dominated by the bound on the right boundary terms, $B_{2,N}=\Theta(N^{-2})$, since
\begin{equation*}
\begin{split}
    \E \sigmasqcvest &= \frac{1}{2(H+1)(2H+1)} B_{1,N} + \frac{1}{2(H+1)(2H+1)} I_N + \frac{1}{(2H+1)} B_{2,N} \\
        &= \bigo(N^{-2-2H}) + \Theta(N^{-1-2H}) + \Theta(N^{-2})
\end{split}
\end{equation*}
As $\E \sigmasqicvest = I_N/(2(H+1)(2H+1))$, it follows that $\sigmasqicvest = \Theta(N^{-1-2H})$ when $l=1$.
\end{proof}

\subsection{Proofs for \Cref{sec:discussion}}
\label{sec:proofs-discussion}

\begin{proof}[Proof of \Cref{res:uq-theorem-exp}]
  We only provide the proof for the case $l = 1$ and leave the simpler case $l = 0$ to the reader.
  Let $x \in (x_{n-1}, x_n)$.
  From the expression for $m_N$ in \Cref{sec:explicit-post-mean-cov}, we get
  \begin{equation*}
    \begin{split}
      \E[ f(x) - m_N(x) ]^2 &= \E \Bigg[ f(x) - \frac{(x_n - x) f(x_{n-1}) + (x - x_{n-1}) f(x_n)}{\Delta x_{n-1}} \Bigg]^2 \\
      &= \frac{1}{\Delta x_{n-1}^2} \E\big[ (x - x_{n-1})(f(x_n) - f(x)) - (x_n - x)(f(x) - f(x_{n-1})) \big]^2.
      \end{split}
  \end{equation*}
  Then, we can use~\eqref{eq:l=1-exp-2} with $x_n$ instead of $x_{n+1}$ and $x$ instead of $x_n$ to get
  \begin{equation*}
      \E[ f(x) - m_N(x) ]^2 = \frac{(x_n - x)(x - x_{n-1})}{C_H \Delta x_{n-1}} \big[ \Delta x_{n-1}^{2H+1} - (x_n - x)^{2H+1} - (x - x_{n-1})^{2H+1} \big],
  \end{equation*}
  where $C_H = 2(H+1)(2H+1)$.
  The expression for $k_N$ in \Cref{sec:explicit-post-mean-cov} gives
  \begin{equation*}
    \frac{\E[ f(x) - m_N(x) ]^2}{k_N(x)} = \frac{1}{C_H} \big[ \Delta x_{n-1}^{2H+1} - (x_n - x)^{2H+1} - (x - x_{n-1})^{2H+1} \big].
  \end{equation*}

  By removing the negative terms and using the quasi-uniformity \eqref{eq:quasi-uniformity-2}, we obtain
  \begin{equation*}
      \sup_{x \in [0, T]} \frac{\E[ f(x) - m_N(x) ]^2}{k_N(x)} \leq \frac{ (T C_\textup{qu})^{2H+1} }{C_H} N^{-1-2H},
  \end{equation*}
  To see that this bound is tight, observe that for the midpoint $x = (x_n + x_{n-1}) / 2$ we have $x_n - x = x - x_{n-1} = \Delta x_{n-1} / 2$ and
  \begin{equation*}
    \frac{ \E[ f(x) - m_N(x) ]^2 }{ k_N(x) } = \frac{1}{ C_H} \bigg(1 - \frac{1}{2^{2H}} \bigg) \Delta x_{n-1}^{2H+1} \geq \frac{T^{2H + 1}}{C_H C_\textup{qu}^{2H + 1}} \bigg(1 - \frac{1}{2^{2H}} \bigg) N^{-1-2H}
  \end{equation*}
  by the quasi-uniformity.
  Therefore
  \begin{equation*}
    \sup_{x \in [0, T]} \frac{\E[ f(x) - m_N(x) ]^2}{k_N(x)} = \Theta(N^{-1-2H})
  \end{equation*}
  %}
  when $l = 1$.
  It can be similarly shown that
  \begin{equation*}
    \sup_{x \in [0, T]} \frac{\E[ f(x) - m_N(x) ]^2}{k_N(x)} = \Theta(N^{1-2H})
  \end{equation*}
  when $l = 0$.
  The claims then follow from the rates for $\E\sigmasqcvest$ and $\E\sigmasqmlest$ in \Cref{res:holder-spaces-exp,res:holder-spaces-exp-ml}.
\end{proof}

\section*{Acknowledgements}
The authors are grateful to the anonymous reviewers and the Associate Editor for their time and effort in reviewing the manuscript, and the comments and suggestions that improved the paper. In particular, one of the reviewers suggested removal of the rightmost boundary point in the CV estimator, which led to the introduction on the interior cross-validation estimator.
MN acknowledges support from the U.K. Research and Innovation under grant
number EP/S021566/1.
MK has been supported by the French government, through the 3IA Cote d’Azur Investment in the Future Project managed by the National Research Agency (ANR) with the reference number ANR-19-P3IA-0002.
TK was supported by the Academy of Finland postdoctoral researcher grant \#338567 ``Scalable, adaptive and reliable probabilistic integration''.
Part of this research was carried out during a visit by TK to EURECOM in May 2023 that was funded by the Institut fran\c{c}ais de Finlande, the Embassy of France to Finland, and the Finnish Society of Sciences and Letters.
MM gratefully acknowledges financial support by the European Research Council through ERC StG Action 757275 / PANAMA; the DFG Cluster of Excellence “Machine Learning - New Perspectives for Science”, EXC 2064/1, project number 390727645; the German Federal Ministry of Education and Research (BMBF) through the T\"{u}bingen AI Center (FKZ: 01IS18039A); and funds from the Ministry of Science, Research and Arts of the State of Baden-W\"{u}rttemberg.
\clearpage

\appendix

\section{Connection between the ML and CV estimators} \label{sec:proofs-ml-cv-relation}

Here we prove a connection between the ML and CV estimators; see \Cref{remark:ml-cv}.
Let
\begin{equation*}
  C(N, p) = \binom{N}{p} = \frac{N!}{p! (N - p)!}
\end{equation*}
denote the binomial coefficient.
The leave-$p$-out cross-validation (LPO-CV) estimator of $\sigma^2$ is
\begin{equation}
\label{eq:leave-p-out-general}
  \hat{\sigma}_{\textup{CV}(p)}^2 = \frac{1}{C(N,p)} \sum_{i=1}^{C(N, p)} \frac{1}{p} \sum_{n=1}^p \frac{[f(x_{p,i,n}) - m_{\setminus \{p, i\}}(x_{p,i,n})]^2}{ k_{\setminus \{p, i\}}(x_{p,i, n})},
\end{equation}
where $i$ indexes the $N$-choose-$p$ possible sets of held-out datapoints, $\bx_{\setminus \{p,i\}}$, among $\bx$ and $n \leq p$ the data points left out of each of these sets. That is, for each $p$ and $i$ we have
\begin{equation*}
  \bx = \bx_{\setminus \{p, i\}} \cup \{x_{p,i,1}, \ldots, x_{p,i,p}\}.
\end{equation*}
The functions $m_{\setminus \{p, i\} }$ and $k_{\setminus \{p, i\} }$ are the GP conditional mean and variance based on the set $\bx_{\setminus \{p,i\}}$, which contains $N-p$ points.
The purpose of this section is to prove that
\begin{equation} \label{eq:ml-lpo-cv-connection}
    \hat{\sigma}_\textup{ML}^2 = \frac{1}{N} \sum_{p=1}^N \hat{\sigma}_{\textup{CV}(p)}^2.
\end{equation}

  Denote $\nu(\bx) = f(\bx)^\top k(\bx, \bx)^{-1} f(\bx)$.
  The block matrix inversion formula applied to $g(\bx_{\setminus\{p,i\}} \cup \{x\})$ and the equations in \Cref{sec:par-est-for-gps} for the conditional mean and variance yield
  \begin{equation} \label{eq:rkhs-norm-recursion}
      \frac{[f(x) - m_{\setminus\{p,i\}}(x)]^2}{k_{\setminus\{p,i\}}(x)} = \nu( \bx_{\setminus\{p,i\}} \cup \{x\} ) - \nu( \bx_{\setminus\{p,i\}} )
  \end{equation}
  for any $1 \leq p \leq N$ and $x \notin \bx_{\{p,i\}}$, where we use the convention $\nu(\bx_{\setminus\{N,i\}}) = \nu(\emptyset) = 0$.
  For each $1 \leq p \leq N$, $i \leq C(N, p)$ and $n \leq p$ there is a unique index $j(p, i, n) \leq C(N, p-1)$ such that
  \begin{equation} \label{eq:set-spawning}
    \bx_{\setminus\{p,i\}} \cup \{x_{p,i,n}\} = \bx_{\setminus\{p-1, j(p,i,n)\}}.
  \end{equation}
  Setting $x = x_{p,i,n}$ in~\eqref{eq:rkhs-norm-recursion} gives
  \begin{equation*}
    \frac{[f(x_{p,i,n}) - m_{\setminus\{p,i\}}(x_{p,i,n})]^2}{k_{\setminus\{p,i\}}(x_{p,i,n})} = \nu( \bx_{\setminus\{p,i\}} \cup\{x_{p,i,n}\}) - \nu( \bx_{\setminus\{p,i\}} ).
  \end{equation*}
  Therefore
  \begin{equation} \label{eq:leave-p-out-average-1}
    \begin{split}
      \sum_{p=1}^N \hat{\sigma}_{\textup{CV}(p)}^2 &= \frac{1}{N}\sum_{p=1}^N \frac{1}{C(N,p)} \sum_{i=1}^{C(N, p)} \frac{1}{p} \sum_{n=1}^p \frac{[f(x_{p,i,n}) - m_{\setminus \{p, i\}}(x_{p,i,n})]^2}{ k_{\setminus \{p, i\}}(x_{p,i, n})} \\
      &= \sum_{p=1}^N \frac{1}{C(N,p)} \sum_{i=1}^{C(N, p)} \frac{1}{p} \sum_{n=1}^p \big[ \nu( \bx_{\setminus\{p,i\}} \cup\{x_{p,i,n}\}) - \nu( \bx_{\setminus\{p,i\}} ) \big].
      \end{split}
  \end{equation}
  By~\eqref{eq:set-spawning} from each set $\bx_{\setminus\{p,i\}}$ on level $p$ (i.e., sets from which $p$ points have been left out), $p$ sets on level $p - 1$ can be obtained by adding one of the left-out datapoints.
  However, there are $C(N,p)$ sets on level $p$ and $C(N, p-1)$ sets on level $p-1$.
  Hence for each set $\bx_{\setminus\{p-1,j\}}$ on level $p-1$ there are
  \begin{equation*}
    p \cdot \frac{C(N,p)}{C(N,p-1)} = p \cdot \frac{N!(p-1)!(N-p+1)!}{N!p!(N-p)!} = N-p+1
  \end{equation*}
  combinations of sets $\bx_{\setminus\{p,i\}}$ on level $p$ and points $x_{p,i,n}$ left out of these sets such that $\bx_{\setminus\{p,i\}} \cup \{x_{p,i,n}\} = \bx_{\setminus\{p-1,j\}}$.
  Therefore
  \begin{equation*}
    \begin{split}
      \sum_{i=1}^{C(N, p)} &\frac{1}{p} \sum_{n=1}^p \big[ \nu( \bx_{\setminus\{p,i\}} \cup\{x_{p,i,n}\}) - \nu( \bx_{\setminus\{p,i\}} ) \big] \\
      &= \sum_{i=1}^{C(N, p)} \frac{1}{p} \sum_{n=1}^p \nu( \bx_{\setminus\{p,i\}} \cup\{x_{p,i,n}\}) - \sum_{i=1}^{C(N, p)} \frac{1}{p} \sum_{n=1}^p \nu( \bx_{\setminus\{p,i\}} ) \\
      &= \frac{N-p+1}{p} \sum_{j=1}^{C(N, p-1)} \nu( \bx_{\setminus\{p-1,j\}}) - \sum_{i=1}^{C(N, p)} \nu( \bx_{\setminus\{p,i\}} )
    \end{split}
  \end{equation*}
  and consequently~\eqref{eq:leave-p-out-average-1} writes
  \begin{equation*}
    \begin{split}
      \sum_{p=1}^N \hat{\sigma}_{\textup{CV}(p)}^2 &= \sum_{p=1}^N \frac{1}{C(N,p)} \Bigg[ \frac{N - p + 1}{p} \sum_{j=1}^{C(N, p-1)} \nu( \bx_{\setminus\{p-1,j\}}) - \sum_{i=1}^{C(N,p)} \nu( \bx_{\setminus\{p,i\}} ) \Bigg] \\
      &= \sum_{p=1}^N \Bigg[ \frac{1}{C(N,p-1)} \sum_{j=1}^{C(N, p-1)} \nu( \bx_{\setminus\{p-1,j\}}) - \frac{1}{C(N,p)} \sum_{i=1}^{C(N,p)} \nu( \bx_{\setminus\{p,i\}} ) \Bigg],
    \end{split}
  \end{equation*}
  which is a telescoping sum.
  We are left with
  \begin{equation*}
    \sum_{p=1}^N \hat{\sigma}_{\textup{CV}(p)}^2 = \frac{1}{C(N, 0)} \sum_{j=1}^{C(N,0)} \nu( \bx_{\setminus\{0,j\}}) - \frac{1}{C(N,N)} \sum_{i=1}^{C(N,N)} \nu( \bx_{\setminus\{N,i\}} ),
  \end{equation*}
  where $\nu( \bx_{\setminus\{0,j\}}) = f(\bx)^\top k(\bx, \bx)^{-1} f(\bx)$ and $\nu( \bx_{\setminus\{N,i\}} ) = \nu(\emptyset) = 0$.
  Thus
  \begin{equation*}
    \frac{1}{N} \sum_{p=1}^N \hat{\sigma}_{\textup{CV}(p)}^2 = \frac{f(\bx)^\top k(\bx, \bx)^{-1} f(\bx)}{N} = \sigmasqmlest,
  \end{equation*}
  which establishes~\eqref{eq:ml-lpo-cv-connection}.

\section{Further discussion on~\Cref{res:fqv-estimator} }
\label{sec:discussion-thm-fqv-estimator}
The requirement of having the same $V^2(f)$ for all sequences of partitions quasi-uniform with constant $2$ can be relaxed somewhat: trivially, it is sufficient that the quadratic variation is $V^2(f)$ specifically with respect to even-points and odd-points sequences of sub-partitions used in the proof in \Cref{sec:proofs-deterministic}. Furthermore, we may even have different quadratic variations with respect to said sequences. Then the results becomes
\begin{equation*} \label{eq:liminf-limsup-v2f-generalisation}
   \lim_{N \to \infty} \hat{\sigma}^2_\mathrm{CV} = \frac{\nu}{T} \qquad \text{for} \qquad \nu = \frac{ V_0^2(f) + V_1^2(f)}{2},
\end{equation*}
where $V_0^2(f)$ and $V_1^2(f)$ are quadratic variations with respect to the even- and odd-points sub-partitions respectively, meaning that
\begin{align*}
  V^2(f) &= \lim_{N \to \infty} \sum_{n=1}^{N-1} (f_{n+1} - f_{n} )^2, \\
  V_0^2(f) &= \lim_{N \to \infty} \sum_{n=1}^{\lfloor \frac{N-2}{2} \rfloor} (f_{2n+2} - f_{2n} )^2, \\
  V_1^2(f) &= \lim_{N \to \infty} \sum_{n=1}^{\lfloor \frac{N-1}{2} \rfloor} (f_{2n+1} - f_{2n-1} )^2.
\end{align*}

\section{Explicit expression for the leave-$p$-out estimator}
\label{sec:explicit_expression_for_leave_p_out}

Using the expressions for posterior mean and covariance functions in \eqref{eq:explicit-post-mean} and~\eqref{eq:explicit-post-cov}, we may derive an explicit expression for the leave-$p$-out cross-validation (LPO-CV) estimator of the scale parameter, given in~\eqref{eq:leave-p-out-general} by
\begin{equation*}
  \hat{\sigma}_{\textup{CV}(p)}^2 = \frac{1}{C(N,p)} \sum_{i=1}^{C(N, p)} \frac{1}{p} \sum_{n=1}^p \frac{[f(x_{p,i,n}) - m_{\setminus \{p, i\}}(x_{p,i,n})]^2}{ k_{\setminus \{p, i\}}(x_{p,i, n})}.
\end{equation*}
The expression is less straightforward than that for $p=1$. Denote by $x_{\lfloor p,i,n \rfloor}$ the largest point in the set $\bx_{\setminus \{p,i\}} = \bx \setminus \{x_{p,i,1}, \ldots, x_{p,i,p}\}$ that does not exceed $x_{p,i,n}$, and by $x_{\lceil p,i,n \rceil}$ the smallest point in the set $\bx_{\setminus \{p,i\}}$ that exceeds $x_{p,i,n}$. Through somewhat cumbersome arithmetic derivations it can be shown that the estimator takes the form
\begin{align*}
  \hat{\sigma}_{\textup{CV}(p)}^2 = \frac{1}{C(N,p)} \sum_{i=1}^{C(N, p)} \bigg[  B_{p,i,1} + \sum_{n=2}^{p-1} I_{p,i,n} + B_{p,i,p} \bigg]
\end{align*}
where, for $\Delta x^-_{p,i,n} = (x_{p,i,n} - x_{\lfloor p,i,n \rfloor})$ and $\Delta x^+_{p,i,n} = (x_{ \lceil p,i,n \rceil } - x_{p,i,n})$, the inner term is
\begin{equation*}
    I_{p,i,n} = \frac{\Delta x^-_{p,i,n}(f_{\lceil p,i,n \rceil} - f_{p,i,n }) - \Delta x^+_{p,i,n} (f_{p,i,n} - f_{\lfloor p,i,n \rfloor})}{ (\Delta x^+_{p,i,n} + \Delta x^-_{p,i,n}) \Delta x^+_{p,i,n} \Delta x^-_{p,i,n}},
\end{equation*}
and the boundary terms $B_{p,i,1}$ and $B_{p,i,p}$ depend on whether the $i$\textsuperscript{th} set contains $x_1$ or $x_N$, respectively. Specifically,
\begin{align*}
    B_{p,i,1} &= \begin{cases}
    \frac{(x_{ \lceil p,i,1 \rceil} f_{p,i,1} -  x_{p,i,1} f_{ \lceil p,i,1 \rceil} )^2}{ x_{ p,i,1 } x_{ \lceil p,i,1 \rceil} \Delta x^+_{p,i,1} } & \text{if the $i$\textsuperscript{th} set contains $x_1$,}\\
    I_{p,i,1} & \text{otherwise,}
    \end{cases} \\
    B_{p,i,p} &= \begin{cases}
    \frac{(f_{p,i,p} - f_{ \lfloor p,i,p \rfloor})^2}{ \Delta x^-_{p,i,p} } & \text{if the $i$'th set contains $x_N$,}\\
    I_{p,i,p} & \text{otherwise.}
    \end{cases}
\end{align*}
Though more cumbersome, it may be feasible to conduct convergence analysis similar to that in~\Cref{sec:limit-behaviour-for-sigma} for $\hat{\sigma}_{\textup{CV}(p)}^2$. We leave this up to future work.

\section{Comparison of CV and ML estimators for Mat\'ern kernels}
\label{sec:cv-vs-ml-sobolev}

A natural next step is extending the analysis to kernels whose reproducing kernel Hilbert spaces (RKHSs) are norm-equivalent to Sobolev spaces, such as the commonly used Mat\'ern kernels. The ML estimator for Mat\'ern kernels was analysed in~\citet{Karvonen2020}.
Their experiments in Section~5.1 suggest that, for $x_+ \coloneqq \max(x, 0)$,
\begin{equation}
\label{eq:karvonen2020mle}
    \sigmasqmlest = \Theta(N^{2(\nu_\mathrm{model} - 2 \nu_\mathrm{true})_+ -1})
\end{equation}
when $k_{\nu_\mathrm{model}}$ is a Mat\'ern kernel of order $\nu_\mathrm{model}$ and $f$ is a finite linear combination of the form $f = \sum_{i=1}^m \alpha_i k_{\nu_\mathrm{true}}(\cdot, x_i)$ for some $m \in \mathbb N$, $\alpha_i \in \mathbb R$, $x_i \in [0, 1]$, and the Mat\'ern kernel $k_{\nu_\mathrm{true}}$ of order $\nu_\mathrm{true}$.
Empirically, we compare this to the rate of the CV estimator in~\Cref{fig:cv-vs-ml-sobolev}. The test functions $f$ are posterior means of a GP with the $k_{\nu_\mathrm{true}}$ kernel conditioned on points $\{(x_1, y_1), \dots, (x_{10}, y_{10})\}$, where each $x_i$ and $y_i$ is sampled i.i.d from the uniform distribution on $[0,1]$. Since such $f$ are of the form $f = \sum_{i=1}^{10} \alpha_i k_{\nu_\mathrm{true}}(\cdot, x_i)$, we expect the MLE rate in~\eqref{eq:karvonen2020mle} to apply; we use experimental data and the results in~\Cref{res:holder-spaces,res:holder-spaces-exp-ml} to hypothesise what the rate in each individual example is.
Similarly to the observations for the Brownian motion kernel, we see that the CV estimator adapts to the smoothness of the true function over a larger ranger of smoothness compared to the ML estimator. For instance, for $\nu_\mathrm{model}=1$, the experimental results suggest that the dependence of rate on $\nu_\mathrm{true}$ is as illustrated in~\Cref{fig:sobolev-fig}. While the CV and the ML estimators adapt to the function smoothness when $\nu_\mathrm{true} \leq 1/2$, for $\nu_\mathrm{true} \in [1/2, 3/4]$ only the CV estimator continues adapting to the smoothness. This implies the CV estimator is less likely to become asymptotically overconfident in the event of undersmoothing.
\begin{figure}
    \centering
    \includegraphics[width=0.8\textwidth]{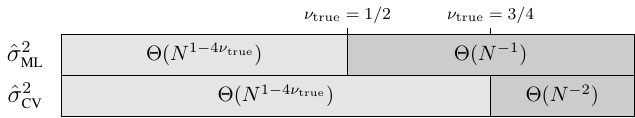}
    \caption{Rates of decay for the ML and CV estimators for the Mat\'ern kernel of order $1$, and a true function that is a linear combination of Mat\'ern kernels of order $\nu_\mathrm{true}$. The ML rate is given in~\citet[Equation 5.2]{Karvonen2020}. The CV rate is empirically observed in~\Cref{fig:cv-vs-ml-sobolev}.
    Observe that the CV estimator's range of adaptation to the smoothness $\nu_\mathrm{true}$ is wider than the ML estimator's.
    }
    \label{fig:sobolev-fig}
\end{figure}

\begin{figure}
    \centering
    \includegraphics[width=\textwidth]{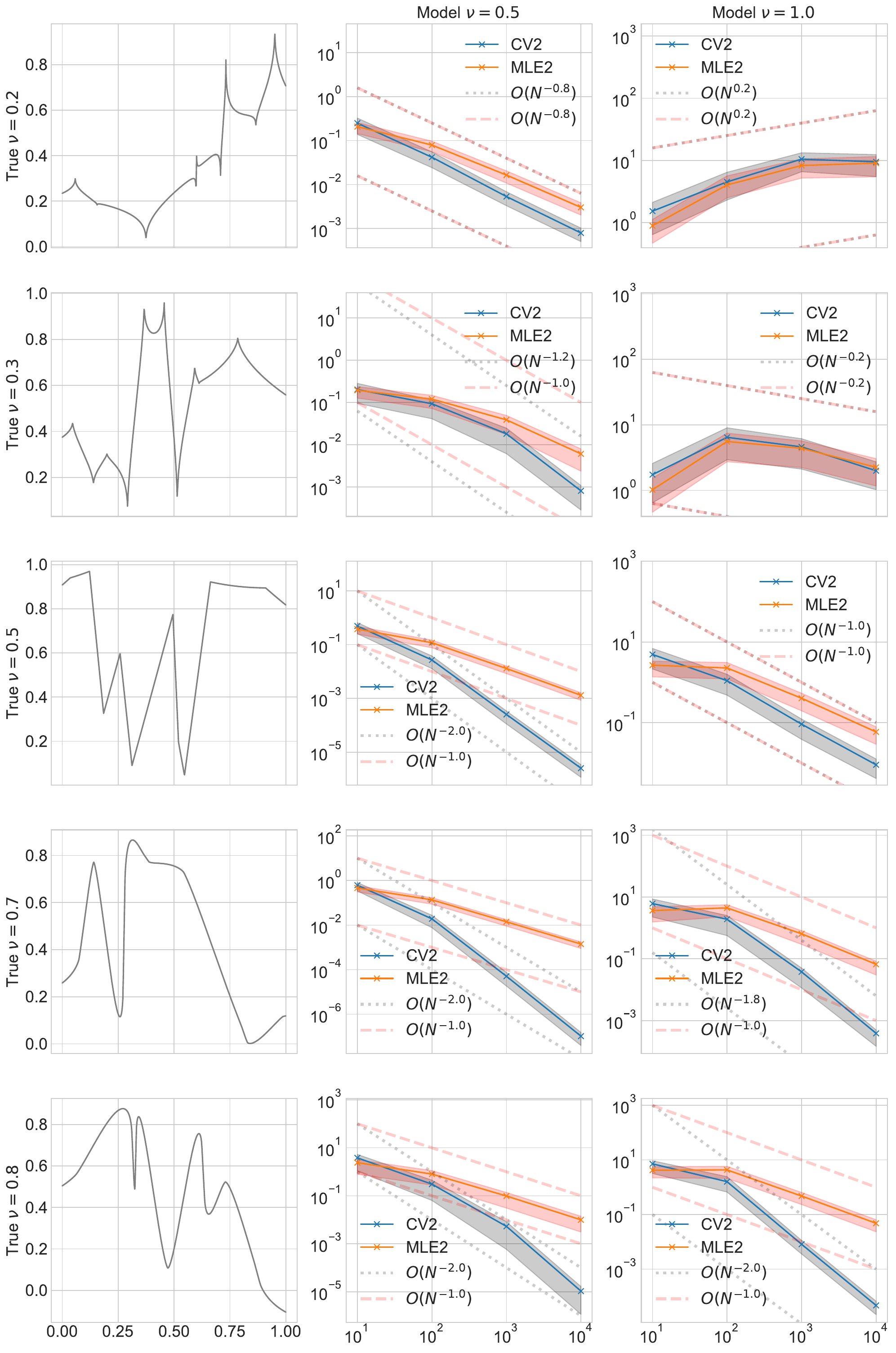}
    \caption{Asymptotics of CV estimator compared to asymptotics of the ML estimator, for the Mat\'ern kernel $\nu_\mathrm{model}$, and a true function that is a finite linear combination of Mat\'ern kernels $\nu_\mathrm{true}$.}
    \label{fig:cv-vs-ml-sobolev}
\end{figure}

\bibliography{references}

\begin{thebibliography}{}

\bibitem[Anderes, 2010]{Anderes2010}
Anderes, E. (2010).
\newblock On the consistent separation of scale and variance for {G}aussian
  random fields.
\newblock {\em The Annals of Statistics}, 38(2):870--893.

\bibitem[Bachoc, 2013]{Bachoc2013}
Bachoc, F. (2013).
\newblock Cross validation and maximum likelihood estimations of
  hyper-parameters of {G}aussian processes with model misspecification.
\newblock {\em Computational Statistics \& Data Analysis}, 66:55--69.

\bibitem[Bachoc, 2018]{bachoc2018asymptotic}
Bachoc, F. (2018).
\newblock Asymptotic analysis of covariance parameter estimation for {G}aussian
  processes in the misspecified case.
\newblock {\em Bernoulli}, 24(2):1531--1575.

\bibitem[Bachoc et~al., 2020]{bachoc2020asymptotic}
Bachoc, F., Betancourt, J., Furrer, R., and Klein, T. (2020).
\newblock Asymptotic properties of the maximum likelihood and cross validation
  estimators for transformed {G}aussian processes.
\newblock {\em Electronic Journal of Statistics}, 14(1):1962--2008.

\bibitem[Bachoc et~al., 2017]{Bachoc2017}
Bachoc, F., Lagnoux, A., and Nguyen, T. M.~N. (2017).
\newblock Cross-validation estimation of covariance parameters under
  fixed-domain asymptotics.
\newblock {\em Journal of Multivariate Analysis}, 160:42--67.

\bibitem[Beck and Guillas, 2016]{beck2016sequential}
Beck, J. and Guillas, S. (2016).
\newblock Sequential design with mutual information for computer experiments
  ({MICE}): Emulation of a tsunami model.
\newblock {\em SIAM/ASA Journal on Uncertainty Quantification}, 4(1):739--766.

\bibitem[{Ben Salem} et~al., 2019]{BenSalem2019}
{Ben Salem}, M., Bachoc, F., Roustant, O., Gamboa, F., and Tomaso, L. (2019).
\newblock {G}aussian process-based dimension reduction for goal-oriented
  sequential design.
\newblock {\em SIAM/ASA Journal on Uncertainty Quantification},
  7(4):1369--1397.

\bibitem[Bevilacqua et~al., 2019]{Bevilacqua2019}
Bevilacqua, M., Faouzi, T., Furrer, R., and Porcu, E. (2019).
\newblock Estimation and prediction using generalized {W}endland covariance
  functions under fixed domain asymptotics.
\newblock {\em The Annals of Statistics}, 47(2):828--856.

\bibitem[Chen et~al., 2021]{Chen2021}
Chen, Y., Owhadi, H., and Stuart, A.~M. (2021).
\newblock Consistency of empirical {B}ayes and kernel flow for hierarchical
  parameter estimation.
\newblock {\em Mathematics of Computation}, 90(332):2527--2578.

\bibitem[Cockayne et~al., 2019]{cockayne2019bayesian}
Cockayne, J., Oates, C.~J., Sullivan, T.~J., and Girolami, M. (2019).
\newblock Bayesian probabilistic numerical methods.
\newblock {\em SIAM Review}, 61(4):756--789.

\bibitem[Cont and Bas, 2023]{ContBas2023}
Cont, R. and Bas, P. (2023).
\newblock Quadratic variation and quadratic roughness.
\newblock {\em Bernoulli}, 29(1):496--522.

\bibitem[Diaconis, 1988]{Diaconis1988}
Diaconis, P. (1988).
\newblock Bayesian numerical analysis.
\newblock In {\em Statistical Decision Theory and Related Topics IV}, volume~1,
  pages 163--175. Springer-Verlag New York.

\bibitem[Du et~al., 2009]{Du2009}
Du, J., Zhang, H., and Mandrekar, V.~S. (2009).
\newblock Fixed-domain asymptotic properties of tapered maximum likelihood
  estimators.
\newblock {\em The Annals of Statistics}, 37(6A):3330--3361.

\bibitem[Dudley, 1973]{10.2307/2959347}
Dudley, R.~M. (1973).
\newblock Sample functions of the {G}aussian process.
\newblock {\em The Annals of Probability}, 1(1):66--103.

\bibitem[Fong and Holmes, 2020]{FongHolmes2020}
Fong, E. and Holmes, C.~C. (2020).
\newblock On the marginal likelihood and cross-validation.
\newblock {\em Biometrika}, 107(2):489--496.

\bibitem[Garnett, 2023]{garnett2023bayesian}
Garnett, R. (2023).
\newblock {\em Bayesian Optimization}.
\newblock Cambridge University Press.

\bibitem[Ginsbourger and Sch{\"a}rer, 2021]{ginsbourger2021fast}
Ginsbourger, D. and Sch{\"a}rer, C. (2021).
\newblock Fast calculation of gaussian process multiple-fold cross-validation
  residuals and their covariances.
\newblock {\em arXiv preprint arXiv:2101.03108}.

\bibitem[Gu et~al., 2018]{gu2018robust}
Gu, M., Wang, X., and Berger, J.~O. (2018).
\newblock Robust gaussian stochastic process emulation.
\newblock {\em The Annals of Statistics}, 46(6A):3038--3066.

\bibitem[Hadji and Szab\'{o}, 2021]{HadjiSzabo2021}
Hadji, A. and Szab\'{o}, B. (2021).
\newblock Can we trust {B}ayesian uncertainty quantification from {G}aussian
  process priors with squared exponential covariance kernel?
\newblock {\em SIAM/ASA Journal on Uncertainty Quantification}, 9(1):185--230.

\bibitem[Hennig et~al., 2015]{hennig2015probabilistic}
Hennig, P., Osborne, M.~A., and Girolami, M. (2015).
\newblock Probabilistic numerics and uncertainty in computations.
\newblock {\em Proceedings of the Royal Society A: Mathematical, Physical and
  Engineering Sciences}, 471(2179):20150142.

\bibitem[Hennig et~al., 2022]{pnbook2022}
Hennig, P., Osborne, M.~A., and Kersting, H.~P. (2022).
\newblock {\em Probabilistic Numerics: Computation as Machine Learning}.
\newblock Cambridge University Press.

\bibitem[Jones et~al., 1998]{jones1998efficient}
Jones, D.~R., Schonlau, M., and Welch, W.~J. (1998).
\newblock Efficient global optimization of expensive black-box functions.
\newblock {\em Journal of Global Optimization}, 13(4):455.

\bibitem[Karvonen, 2021]{karvonen2021estimation}
Karvonen, T. (2021).
\newblock Estimation of the scale parameter for a misspecified gaussian process
  model.
\newblock {\em arXiv preprint arXiv:2110.02810}.

\bibitem[Karvonen, 2023]{Karvonen2022}
Karvonen, T. (2023).
\newblock Asymptotic bounds for smoothness parameter estimates in {G}aussian
  process interpolation.
\newblock {\em SIAM/ASA Journal on Uncertainty Quantification},
  11(4):1225--1257.

\bibitem[Karvonen and Oates, 2023]{KarvonenOates2023}
Karvonen, T. and Oates, C.~J. (2023).
\newblock Maximum likelihood estimation in {G}aussian process regression is
  ill-posed.
\newblock {\em Journal of Machine Learning Research}, 24(120):1--47.

\bibitem[Karvonen et~al., 2020]{Karvonen2020}
Karvonen, T., Wynne, G., Tronarp, F., Oates, C.~J., and Särkkä, S. (2020).
\newblock Maximum likelihood estimation and uncertainty quantification for
  {G}aussian process approximation of deterministic functions.
\newblock {\em SIAM/ASA Journal on Uncertainty Quantification}, 8(3):926--958.

\bibitem[Kaufman and Shaby, 2013]{Kaufman2013}
Kaufman, C.~G. and Shaby, B.~A. (2013).
\newblock The role of the range parameter for estimation and prediction in
  geostatistics.
\newblock {\em Biometrika}, 100(2):473--484.

\bibitem[Kennedy and O'Hagan, 2001]{kennedy2001bayesian}
Kennedy, M.~C. and O'Hagan, A. (2001).
\newblock Bayesian calibration of computer models.
\newblock {\em Journal of the Royal Statistical Society: Series B (Statistical
  Methodology)}, 63(3):425--464.

\bibitem[K{\i}l{\i}{\c{c}}, 2008]{kilicc2008explicit}
K{\i}l{\i}{\c{c}}, E. (2008).
\newblock Explicit formula for the inverse of a tridiagonal matrix by backward
  continued fractions.
\newblock {\em Applied Mathematics and Computation}, 197(1):345--357.

\bibitem[Loh, 2005]{Loh2005}
Loh, W.-L. (2005).
\newblock Fixed-domain asymptotics for a subclass of {M}atérn-type {G}aussian
  random fields.
\newblock {\em The Annals of Statistics}, 33(5):2344--2394.

\bibitem[Loh and Kam, 2000]{LohKam2000}
Loh, W.-L. and Kam, T.-K. (2000).
\newblock Estimating structured correlation matrices in smooth {G}aussian
  random field models.
\newblock {\em Annals of Statistics}, 28(3):880--904.

\bibitem[Loh and Sun, 2023]{LohSun2023}
Loh, W.-L. and Sun, S. (2023).
\newblock Estimating the parameters of some common {G}aussian random fields
  with nugget under fixed-domain asymptotics.
\newblock {\em Bernoulli}, 29(3):2519--2534.

\bibitem[Loh et~al., 2021]{LohSunWen2021}
Loh, W.-L., Sun, S., and Wen, J. (2021).
\newblock On fixed-domain asymptotics, parameter estimation and isotropic
  {G}aussian random fields with {M}at{\'e}rn covariance functions.
\newblock {\em The Annals of Statistics}, 49(6):3127--3152.

\bibitem[Mallik, 2001]{mallik2001inverse}
Mallik, R.~K. (2001).
\newblock The inverse of a tridiagonal matrix.
\newblock {\em Linear Algebra and its Applications}, 325(1--3):109--139.

\bibitem[Mandelbrot, 1982]{mandelbrot1982fractal}
Mandelbrot, B.~B. (1982).
\newblock {\em The Fractal Geometry of Nature}.
\newblock WH Freeman New York.

\bibitem[M{\"o}rters and Peres, 2010]{morters2010brownian}
M{\"o}rters, P. and Peres, Y. (2010).
\newblock {\em Brownian motion}, volume~30.
\newblock Cambridge University Press.

\bibitem[Nourdin, 2012]{Nourdin2012}
Nourdin, I. (2012).
\newblock {\em Selected aspects of fractional {B}rownian motion}.
\newblock Number~4 in Bocconi \& Springer Series. Springer.

\bibitem[O'Hagan, 1978]{OHagan1978}
O'Hagan, A. (1978).
\newblock Curve fitting and optimal design for prediction.
\newblock {\em Journal of the Royal Statistical Society. Series B
  (Methodological)}, 40(1):1--42.

\bibitem[O’Hagan, 2006]{o2006bayesian}
O’Hagan, A. (2006).
\newblock Bayesian analysis of computer code outputs: A tutorial.
\newblock {\em Reliability Engineering \& System Safety}, 91(10-11):1290--1300.

\bibitem[Petit, 2023]{Petit2023}
Petit, S. (2023).
\newblock Maximum likelihood estimation and prediction error for a {M}atérn
  model on the circle.
\newblock {\em arXiv:2209.07791v5}.

\bibitem[Petit et~al., 2022]{petit2021gaussian}
Petit, S., Bect, J., Feliot, P., and Vazquez, E. (2022).
\newblock Parameter selection in {G}aussian process interpolation: {A}n
  empirical study of selection criteria.
\newblock {\em arXiv:2107.06006v4}.

\bibitem[Rasmussen and Williams, 2006]{rassmussen2006gaussian}
Rasmussen, C.~E. and Williams, C.~K. (2006).
\newblock {\em Gaussian processes for machine learning}, volume~2.
\newblock MIT press Cambridge, MA.

\bibitem[Sacks et~al., 1989]{sacks1989design}
Sacks, J., Welch, W.~J., Mitchell, T.~J., and Wynn, H.~P. (1989).
\newblock Design and analysis of computer experiments.
\newblock {\em Statistical Science}, 4(4):409--423.

\bibitem[Schaback, 2018]{Schaback2018}
Schaback, R. (2018).
\newblock Superconvergence of kernel-based interpolation.
\newblock {\em Journal of Approximation Theory}, 235:1--19.

\bibitem[Shahriari et~al., 2015]{shahriari2015taking}
Shahriari, B., Swersky, K., Wang, Z., Adams, R.~P., and De~Freitas, N. (2015).
\newblock Taking the human out of the loop: A review of bayesian optimization.
\newblock {\em Proceedings of the IEEE}, 104(1):148--175.

\bibitem[Sniekers and van~der Vaart, 2015]{sniekers2015adaptive}
Sniekers, S. and van~der Vaart, A. (2015).
\newblock {Adaptive Bayesian credible sets in regression with a Gaussian
  process prior}.
\newblock {\em Electronic Journal of Statistics}, 9(2):2475 -- 2527.

\bibitem[Sniekers and van~der Vaart, 2020]{sniekers2020adaptive}
Sniekers, S. and van~der Vaart, A. (2020).
\newblock Adaptive bayesian credible bands in regression with a gaussian
  process prior.
\newblock {\em Sankhya A}, 82(2):386--425.

\bibitem[Stein, 1990]{Stein1990}
Stein, M.~L. (1990).
\newblock A comparison of generalized cross validation and modified maximum
  likelihood for estimating the parameters of a stochastic process.
\newblock {\em The Annals of Statistics}, 18(3):1139--1157.

\bibitem[Stein, 1993]{Stein1993}
Stein, M.~L. (1993).
\newblock Spline smoothing with an estimated order parameter.
\newblock {\em The Annals of Statistics}, 21(3):1522--1544.

\bibitem[Stein, 1999]{Stein1999}
Stein, M.~L. (1999).
\newblock {\em Interpolation of Spatial Data: {S}ome Theory for Kriging}.
\newblock Springer Series in Statistics. Springer.

\bibitem[Sundararajan and Keerthi, 2001]{sundararajan2001predictive}
Sundararajan, S. and Keerthi, S.~S. (2001).
\newblock Predictive approaches for choosing hyperparameters in {G}aussian
  processes.
\newblock {\em Neural Computation}, 13(5):1103--1118.

\bibitem[Szab\'{o} et~al., 2015]{Szabo2015}
Szab\'{o}, B., van~der Vaart, A.~W., and van Zanten, J.~H. (2015).
\newblock Frequentist coverage of adaptive nonparametric {B}ayesian credible
  sets.
\newblock {\em The Annals of Statistics}, 43(4):1391--1428.

\bibitem[van~der Vaart and van Zanten, 2008]{VaartZanten2008}
van~der Vaart, A.~W. and van Zanten, J.~H. (2008).
\newblock {\em Reproducing Kernel {H}ilbert spaces of {G}aussian Priors},
  volume~3 of {\em IMS Collections}, pages 200--222.
\newblock Institute of Mathematical Statistics.

\bibitem[Wahba, 1990]{Wahba1990}
Wahba, G. (1990).
\newblock {\em Spline Models for Observational Data}.
\newblock Number~59 in CBMS-NSF Regional Conference Series in Applied
  Mathematics. Society for Industrial and Applied Mathematics.

\bibitem[Wang and Loh, 2011]{WangLoh2011}
Wang, D. and Loh, W.-L. (2011).
\newblock On fixed-domain asymptotics and covariance tapering in {G}aussian
  random field models.
\newblock {\em Electronic Journal of Statistics}, 5:238--269.

\bibitem[Wang, 2007]{wang_almost-sure_2007}
Wang, W. (2007).
\newblock Almost-sure path properties of fractional {Brownian} sheet.
\newblock {\em Annales de l'Institut Henri Poincare (B) Probability and
  Statistics}, 43(5):619--631.

\bibitem[Wang, 2021]{Wang2021}
Wang, W. (2021).
\newblock On the inference of applying {G}aussian process modeling to a
  deterministic function.
\newblock {\em Electronic Journal of Statistics}, 15(2):5014--5066.

\bibitem[Wendland, 2005]{Wendland2005}
Wendland, H. (2005).
\newblock {\em Scattered Data Approximation}.
\newblock Number~17 in Cambridge Monographs on Applied and Computational
  Mathematics. Cambridge University Press.

\bibitem[Wenzel et~al., 2021]{wenzel2021novel}
Wenzel, T., Santin, G., and Haasdonk, B. (2021).
\newblock A novel class of stabilized greedy kernel approximation algorithms:
  Convergence, stability and uniform point distribution.
\newblock {\em Journal of Approximation Theory}, 262:105508.

\bibitem[Wynne et~al., 2021]{wynne2021convergence}
Wynne, G., Briol, F.-X., and Girolami, M. (2021).
\newblock Convergence guarantees for {G}aussian process means with misspecified
  likelihoods and smoothness.
\newblock {\em Journal of Machine Learning Research}, 22(123):1--40.

\bibitem[Xu and Stein, 2017]{XuStein2017}
Xu, W. and Stein, M.~L. (2017).
\newblock Maximum likelihood estimation for a smooth {G}aussian random field
  model.
\newblock {\em SIAM/ASA Journal on Uncertainty Quantification}, 5(1):138--175.

\bibitem[Ying, 1991]{Ying1991}
Ying, Z. (1991).
\newblock Asymptotic properties of a maximum likelihood estimator with data
  from a {G}aussian process.
\newblock {\em Journal of Multivariate Analysis}, 36(2):280--296.

\bibitem[Ying, 1993]{Ying1993}
Ying, Z. (1993).
\newblock Maximum likelihood estimation of parameters under a spatial sampling
  scheme.
\newblock {\em The Annals of Statistics}, 21(3):1567--1590.

\bibitem[Zhang, 2004]{Zhang2004}
Zhang, H. (2004).
\newblock Inconsistent estimation and asymptotically equal interpolations in
  model-based geostatistics.
\newblock {\em Journal of the American Statistical Association},
  99(465):250--261.

\end{thebibliography}

\end{document}